\documentclass[a4paper, 11pt, reqno, twoside]{amsart}
\numberwithin{equation}{section}

\usepackage{amssymb}
\usepackage[T1]{fontenc}
\usepackage[latin1]{inputenc}

\usepackage[pdftex]{graphicx}

\usepackage[usenames,dvipsnames,svgnames,table]{xcolor}

\newcommand{\N}{\mathbb{N}}
\newcommand{\R}{\mathbb{R}}
\newcommand{\s}{\mathbb{S}}
\newcommand{\Z}{\mathbb{Z}}
\newcommand{\ve}{{\varepsilon}}
\newcommand{\pa}{{\partial}}

\newcommand{\BB}{{\mathcal{B}}}

\newcommand{\FF}{{\mathcal{F}}}
\newcommand{\LL}{{\mathcal{L}}}
\newcommand{\OO}{{\mathcal{O}}}
\newcommand{\MM}{{\mathcal{M}}}
\newcommand{\NN}{{\mathcal{N}}}
\newcommand{\RR}{{\mathcal{R}}}
\renewcommand{\SS}{{\mathcal{S}}}
\newcommand{\TT}{{\mathcal{T}}}
\newcommand{\UU}{{\mathcal{U}}}
\newcommand{\VV}{{\mathcal{V}}}

\DeclareMathOperator{\Diff}{D}

\DeclareMathOperator{\sgn}{sgn}
\DeclareMathOperator*{\range}{ran}
\DeclareMathOperator{\spn}{span}

\newtheorem{theorem}{Theorem}[section]
\newtheorem{lemma}[theorem]{Lemma}

\newtheorem{proposition}[theorem]{Proposition}

\theoremstyle{definition}

\newtheorem{example}[theorem]{Example}

\usepackage[colorlinks=true]{hyperref}
\hypersetup{urlcolor=blue, citecolor=NavyBlue, linkcolor=Black}

\renewcommand{\SS}{\mathcal{S}}

\title{Trimodal steady water waves}

\author{Mats Ehrnstr\"om}
\address{Department of Mathematical Sciences, Norwegian University of Science and Technology, 7491 Trondheim, Norway}
\email{mats.ehrnstrom@math.ntnu.no}

\author{Erik Wahl\'en}
\address{Centre for Mathematical Sciences, Lund University, PO 
Box 118, 221\,00 Lund, Sweden}
\email{erik.wahlen@math.lu.se}


\begin{document}

\begin{abstract}
We construct three-dimensional families of small-amplitude gravity-driven rotational steady water waves on finite depth. The solutions contain counter-currents and multiple crests in each minimal period. Each such wave generically is a combination of three different Fourier modes, giving rise to a rich and complex  variety of wave patterns. The bifurcation argument is based on a blow-up technique, taking advantage of three parameters associated with the vorticity distribution, the strength of the background stream, and the period of the wave.
\end{abstract}

\maketitle

\section{Introduction}\label{sec:introduction}

An explicit example of a rotational wave in water of infinite depth was constructed by Gerstner already in 1802 \cite{Gerstner09} (see also \cite{MR2867413} for a more modern treatment of this topic). In spite of this, irrotational flows continued to draw much attention, and the first existence proof for small-amplitude solutions with a general distribution of vorticity, due to Dubreil-Jacotin, was published first in 1934 \cite{0010.22702}.  Only with the paper \cite{MR2027299} by Constantin \& Strauss in 2004 a general result for large-amplitude waves appeared. Using bifurcation and degree theory they constructed a global continuum of waves with an arbitrary vorticity distribution. Since then, much attention has been brought to the richness of phenomena appearing in rotational flows.

Whereas existence proofs for irrotational waves typically involve a perturbation argument starting from a still, or uniform, stream, a general vorticity distribution allows for perturbing very involved background streams, including such with stagnation and so-called critical layers. Critical layers, which are regions of the fluid consisting entirely of closed streamlines,  can be ruled out in the case of irrotational currents; such waves admit at most one stagnation point, which can be located only at the crest of the wave \cite{MR2273961}. In contrast, there are classes of rotational currents allowing for arbitrarily many internal stagnation points and critical layers, the exact number depending on the vorticity distribution and the values of certain parameters.

The first rigorous existence proof for exact steady water waves with a critical layer was given in \cite{Wahlen09}, soon to be followed by the investigation \cite{CV09} by Constantin and Varvaruca. Both investigations used bifurcation theory by  Crandall and Rabinowitz \cite{MR0288640} to establish the existence of small-amplitude waves of constant vorticity. The paper \cite{Wahlen09} is a rigorous justification of linear waves with critical layers found in \cite{MR2409513} as a natural extension of the exact waves constructed in \cite{MR2027299}; it may be extended to more general vorticity distributions. The theory developed in~\cite{CV09} focuses on the case of constant vorticity and encompasses free surfaces that are not graphs of functions, thereby allowing for the possibility of overhanging waves. 

Since linear shear flows (constant vorticity) allow for at most one critical layer, the authors of \cite{MR2964741} and \cite{eew10} considered flows in which the relative horizontal velocity is oscillating. As it turns out, an affine vorticity distribution is enough to induce several critical layers. The solution classes for such vorticity distributions were investigated in \cite{MR2964741}. The corresponding existence theory was pursued in \cite{eew10}. In that paper, a construction of bimodal waves with critical layers and two crests in each minimal period is contained, a class of waves previously known to exist only in the irrotational setting (see \cite{BDT00} and \cite{Jones89,JonesToland86, JonesToland85}). A later contribution to the theory of bimodal waves is \cite{MM13}.

The current paper answers the following natural question appearing in \cite{eew10}. 
Does higher-dimensional bifurcation occur---i.e., can one find additional parameters in order to obtain $n$-modal waves, where $n\ge 3$?

As it turns out, the wave-length parameter can be made to play exactly this desired role (the other parameters appearing naturally in the problem are too closely related, cf. \cite[Remark 4.10]{eew10}), and we are able to answer this question affirmatively. The situation is, however, much more complicated than in lower-dimensional bifurcation, and a precise investigation of the different solution sets in necessary to understand the possible combinations of different wave modes that may appear (see Section~\ref{sec:solutions}). Whether one could find waves of even more complicated form---in the shape, for example, of quadmodal surface profiles---is left as an open problem. The current setting seems not to allow for this, but it would most probably be possible to add surface tension or stratification to achieve this goal, if desired.  

Some comments on related work: A different continuation of the theory on waves with critical layers is the consideration of arbitrary vorticity distributions while still allowing for internal stagnation. Such an investigation was initiated by Kozlov and Kuznetsov in \cite{KK11}, where they examine and classify the running streams (with flat, but free, surfaces) induced by Lipschitz-continuous vorticity distributions. In~\cite{KK10} the same authors give several a priori bounds for waves associated with such a general vorticity distribution, and in \cite{1207.5181} they describe the local bifurcation of waves with critical layers and a general vorticity distribution. An additional recent contribution in this field is the work by Shatah, Walsh and Zeng \cite{1211.3314}, in which exact travelling water waves with compactly supported vorticity distribution are constructed, a new addition in the theory of waves with critical layers. 

The above works on exact waves with critical layers are all on gravity-driven waves in water of finite depth and constant density. A theory for rotational waves in the presence of stratification --- i.e. for a fluid of non-constant density --- was  recently developed by Walsh \cite{MR2529956}. In the investigation \cite{EMM11} Escher, Matioc and Matioc combine ideas from \cite{Wahlen09} and \cite{MR2529956}  to prove the existence  of stratified waves with a linear density distribution and critical layers. This yields the existence of waves whose properties are not that distant from those of constant vorticity, but that may still include a second critical layer. Two further developments in this direction, establishing waves of larger amplitude, are \cite{MB2013} and \cite{HMA2013}.

One of the interesting open questions in this field is the existence of waves with an even more complicated surface profile than the ones constructed in this or any of the above-mentioned papers. In the case of constant vorticity, numerical studies \cite{KoStrauss08, OkamotoShoji01, Teles-daSilvaPeregrine88, Vanden-Broeck_1996} indicate the existence of large-amplitude waves with overhanging profiles and critical layers, as well as non-symmetric surface profiles with several peaks. So far there are no rigorous results confirming these intriguing phenomena.

The outline of this work is as follows: In Section~\ref{sec:formulation} we present the problem in a setting appropriate for our aims, and in Section~\ref{sec:functional} the corresponding functional-analytic framework. Section~\ref{sec:threedim} contains the construction of the actual three-dimensional kernels, and Section~\ref{sec:bifurcation} the Lyapunov--Schmidt reduction and following proof of trimodal gravity water waves. Finally, in Section~\ref{sec:solutions} we investigate the structure of solution sets appearing in the different Fourier regimes (the waves obtained depend on which modes interact). A few illustrations and a concrete example showing the complexity of the waves are given.

\section{Mathematical formulation}\label{sec:formulation}
We consider steady (travelling) waves in two dimensions. Let $B = \{ y = 0\}$ denote the flat bed and $S := \{ y = 1 + \eta(x)\}$ the free surface. The fluid domain is naturally defined by $\Omega := \{ (x,y)\in \R^2 \colon 0 < y < 1 + \eta(x) \}$, and the steady water-wave problem is to find a stream function $\psi$ such that  
\begin{subequations}
\label{eqs:problem}
\begin{equation}
\begin{aligned}
\{ \Delta \psi, \psi \} &= 0 \qquad &&\text{in} \: &\Omega,\\
\psi &= m_0 \qquad &&\text{on} \: &B, \\
\psi &= m_1 \qquad &&\text{on} \: &S, \\
{\textstyle \frac{1}{2}} |\nabla \psi|^2+\eta &=Q \qquad &&\text{on} \: &S.  
\end{aligned}
\end{equation}
Here $\{f,g\} := f_x g_y - f_y g_x$ denotes the Poisson bracket, $\Delta := \partial_x^2 + \partial_y^2$ is the Laplace operator, and $m_0$, $m_1$ and $Q$ are arbitrary constants. The formulation~\eqref{eqs:problem} is equivalent to the Euler equations for a gravity-driven, inviscid, incompressible fluid of constant density and finite depth \cite{MR2964741}, and it allows for both rotation, i.e. $\Delta \psi \neq 0$, and stagnation, i.e. $\nabla \psi = 0$. We consider the case
\begin{equation}
\Delta \psi = \alpha \psi \qquad \text{in} \quad \Omega,
\end{equation}
\end{subequations}
when the vorticity distribution is linear in $\psi$; affine distributions may be incorporated by translation of $\psi$.

Waves exist for any wavelength $\kappa$ (see \cite{eew10}), so we choose to normalise the period to $2\pi$ by a transformation $x \mapsto \kappa x$. In a similar fashion we map the vertical variable onto one of unit range, so that
\[
(x,y) \quad \mapsto \quad (q,s) := \left(\kappa x, \frac{y}{1 + \eta(\kappa x)} \right)
\]
describes the transformation of $\Omega$ onto the strip $\hat \Omega := \{ (q,s) \in \R^2 \colon s \in (0,1)\}$. Let $\hat \psi(q,s) := \psi(x,y)$. Since all $x$-derivates in \eqref{eqs:problem} appear in pairs, it is natural to introduce a wavelength parameter $\xi := \kappa^2$. In the new coordinates the problem~\eqref{eqs:problem} takes the form
\begin{equation}
\label{eq:flatproblem}
\begin{aligned}
\xi \left(\hat \psi_{q}-\frac{s\eta_q\hat\psi_s}{1+\eta}\right)_q - \frac{\xi  s\eta_q}{1+\eta}\left(\hat \psi_{q}-\frac{s\eta_q\hat\psi_s}{1+\eta}\right)_s+\frac{\hat\psi_{ss}}{(1+\eta)^2} &= \alpha \hat \psi  &&\text{in }   \hat \Omega,\\
\hat \psi &= m_0  &&\text{on } s=0, \\
\hat \psi &= m_1  &&\text{on }  s=1, \\
\frac{\xi}{2} \left(\hat\psi_q-\frac{s\eta_q\hat\psi_s}{1+\eta}\right)^2+\frac{\hat \psi_s^2}{2(1+\eta)^2}+\eta &= Q && \text{on }  s=1.  
\end{aligned}
\end{equation}
We pose the problem for $\eta \in C^{2+\beta}_{\text{even}}(\s,\R)$ and $\psi \in C^{2+\beta}_{\text{per,even}}(\overline \Omega, \R)$, where the subscripts \emph{per} and \emph{even} denotes $2\pi$-periodicity and evenness in the horizontal variable, and we have identified $2\pi$-periodic functions with functions defined on the unit circle $\s$. We also require that $\min \eta > -1$. The parameter $\alpha$ influences the nature of $\psi$ in a substantial way (see \cite{MR2964741}), and in order to obtain the desired triply-periodic waves we shall assume that $\alpha<0$.

\subsubsection*{Laminar flows}
Laminar flows are simultaneous solutions of~\eqref{eqs:problem} and~\eqref{eq:flatproblem} for which 
$\eta=0$ and $\psi$ is independent of $q$. They are given 
by the formula 
\begin{equation}
\label{eq:trivial}
\psi_0(\cdot;\mu, \lambda, \alpha) := \mu \cos(\theta_0 (\cdot-1)+\lambda),
\end{equation}
where $\mu, \lambda \in \R$ are arbitrary constants and 
\[
\theta_k:=|\alpha+\xi k^2|^{1/2}, \qquad k \in \Z.
\]
The values of $Q=Q(\mu, \alpha, \lambda)$, $m_0=m_0(\mu, \alpha, \lambda)$ 
and  $m_1=m_1(\mu,\lambda)$ 
are determined from~\eqref{eqs:problem}, i.e.
\begin{equation}\label{eq:Q}
Q(\mu, \lambda, \alpha) := \frac{\mu^2\theta_0^2\sin^2(\lambda)}{2}
\end{equation}
and
\[
m_1(\mu, \lambda) := \mu\cos(\lambda), \quad m_0(\mu, \alpha, \lambda) := \mu\cos(\lambda-\theta_0)
\]
Since the laminar solutions $\psi_0$ are the same in the variables $(q,s)$ and $(x,y)$, they are also independent of the wavelength parameter $\xi$.

\section{Functional-analytic framework}\label{sec:functional}
In this section we describe the framework developed in \cite{Wahlen09} and \cite{eew10}, which shall be used for the three-dimensional bifurcation. Note that the theory in \cite{eew10} is written in the variables $(x,s)$, whereas here we use the wavelength parameter $\xi$. 

\subsubsection*{The map $\FF$}
We shall linearise the problem \eqref{eq:flatproblem} around a laminar flow $\psi_0$, whence we introduce a disturbance $\hat \phi$ through $\hat\psi = \psi_0 + \hat \phi$, and the function space 
\[
X := X_1 \times X_2 :=  C^{2+\beta}_{\text{even}}(\s) \times \Bigl\{ \hat \phi \in C_\text{per, even}^{2+\beta}(\overline{\hat \Omega}) \colon \hat \phi|_{s=1} = \hat \phi|_{s=0}=0 \Bigr\}.
\]
Furthermore, we define the target space 
\[
Y := Y_1\times Y_2 := C_\text{even}^{1 +\beta}(\s) \times  C_\text{per, even}^{\beta}(\overline{\hat \Omega}),
\]
and the sets
\[
\OO := \{ (\eta,\hat \phi) \in X \colon \min \eta > -1 \}
\]
and
\[
\UU := \{ (\mu, \alpha, \lambda, \xi) \in \R^4 \colon \mu \ne 0, \alpha < 0, \sin(\lambda)\ne 0, \xi > 0 \},
\]
which conveniently captures all necessary assumptions on the parameters (cf. \cite{eew10}).
Then $\OO \subset X$ is an open neighbourhood of the origin in $X$, and the embedding $X \hookrightarrow Y$ is compact. Elements of $Y$ will be written $w := (\eta,\hat \phi)$ and
elements of $\UU$ will be written $\Lambda :=(\mu, \alpha, \lambda, \xi)$, and $\Lambda^\prime := (\mu, \alpha, \lambda)$ to indicate independence of $\xi$. 
Define the operator 
$\FF\colon \OO \times \UU \to Y$ by
\[
\FF(w, \Lambda) := (\FF_1(w,\Lambda), \FF_2(w,\Lambda))
\]
where  
\begin{align*}
\FF_1(w, \Lambda) &:=
\frac12\left[ 
\xi \left(\hat \phi_q-\frac{s\eta_q((\psi_0)_s(s;\Lambda^\prime)+\hat\phi_s)}{1+\eta}\right)^2+\frac{((\psi_0)_s(s;\Lambda^\prime)+\hat \phi_s)^2}{(1+\eta)^2}\right]_{s=1}\\
&\quad +\eta-Q(\Lambda),
\end{align*}
and
\begin{align*}
\FF_2(w,\Lambda) &:=
\xi \left(\hat \phi_{q}-\frac{s\eta_q((\psi_0)_s(s;\Lambda^\prime)+\hat\phi_s)}{1+\eta}\right)_q\\
&\quad -\frac{\xi s\eta_q}{1+\eta}\left(\hat \phi_{q}-\frac{s\eta_q((\psi_0)_s(s;\Lambda^\prime)+\hat\phi_s)}{1+\eta}\right)_s\\
&\quad +\frac{(\psi_0)_{ss}(s;\Lambda^\prime)+\hat\phi_{ss}}{(1+\eta)^2} - \alpha \left( \psi_0(s; \Lambda^\prime) + \hat \phi \right).
\end{align*}
The problem $\FF((\eta,\hat \phi),\Lambda)= 0$, $(\eta,\hat\phi) \in \OO$,  is then equivalent to the water-wave problem \eqref{eqs:problem}, the map $(\eta,\hat\phi) \mapsto (\eta,\hat\psi)$ is continuously differentiable, and $\FF$ is real analytic $\OO \times \UU \to Y$ \cite[Lemma~3.1]{eew10}.

\subsubsection*{Linearization}
The Fr\'echet derivative of $\FF$ at $w=0$ is given by the pair
\begin{align}
\label{eq:DF1}
\Diff_w \FF_{1}(0,\Lambda^\prime)w&=\left[(\psi_0)_s\, \hat\phi_s - (\psi_0)_s^2\, \eta + \eta\right]\Big|_{s=1},\\
\label{eq:DF2}
\Diff_w \FF_{2}(0,\Lambda)w &= \xi \hat \phi_{qq}+\hat \phi_{ss}-\alpha \hat \phi- \xi s(\psi_0)_s\eta_{qq}-2(\psi_0)_{ss}\eta.
\end{align}
Let $\tilde X_2 := \Bigl\{ \phi \in C_\text{per, even}^{2+\beta}(\overline{\hat \Omega}) \colon  \phi|_{s=0}=0 \Bigr\}$ and $\tilde X := \left\{ (\eta,\hat \phi) \in X_1 \times \tilde X_2 \right\}$.

\begin{lemma}[\cite{eew10}]
\label{lemma:T}
The bounded, linear operator $ \TT(\Lambda^\prime) \colon \tilde X_2\to X$ given by
\[
 \TT(\Lambda^\prime) \phi := \left(-\frac{\phi|_{s=1}}{(\psi_0)_s(1)},\phi - \frac{s(\psi_0)_s\, \phi|_{s=1}}{(\psi_0)_s(1)}\right)
\]
is an isomorphism. Define $\LL(\Lambda) := \Diff_w \FF(0,\Lambda)  \TT(\Lambda^\prime)\colon \tilde X_2 \to Y$. Then
\begin{equation}
\label{eq:Lexpression}
\LL(\Lambda)\phi=\left(\, \left[(\psi_0)_s \phi_s- \left((\psi_0)_{ss}+\frac{1}{(\psi_0)_s}\right)\phi\right]\biggl|_{s=1}, \,\,  (\xi \pa_q^2+\pa_s^2 -\alpha)\phi\, \right).
\end{equation}
\end{lemma}

When the dependence on the parameters is unimportant, we shall for convenience refer to $\Diff_w \FF(0,\Lambda)$, $\LL(\Lambda)$ and $\TT(\Lambda)$ simply as $\Diff_w \FF(0)$, $\LL$ and $\TT$. Via $\TT$, elements $\phi \in \tilde X_2$ can be ``lifted'' to elements $(\eta, \hat \phi) \in \tilde X$ using the correspondence induced by the first component of $\TT\phi$.

\begin{lemma}[\cite{eew10}]
\label{lemma:P}
The mapping $\eta_{(\cdot)}$ defined by
\[
\tilde X_2 \ni \phi \: \stackrel{\eta_{(\cdot)}}{\mapsto} \: \eta_\phi = -\frac{\phi|_{s=1}}{(\psi_0)_s(1)} \in X_1,
\]
is linear and bounded, whence $\phi \mapsto (\eta_\phi,\phi)$ is linear and bounded $\tilde X_2 \to \tilde X$.
\end{lemma}

Any property of the operator $\Diff_w \FF(0)$ can be conveniently studied using the operator $\LL$. In particular, since $\range \Diff_w \FF(0)=\range \LL$ and $\ker \Diff_w \FF(0)= \TT\ker \LL$, the following lemma shows that $\Diff_w \FF(0)\colon X \to Y$ is Fredholm of index $0$. 

\begin{lemma}[\cite{eew10}]
\label{lemma:Lproperties}
The operator $\LL\colon \tilde X_2\to Y$ is Fredholm of index $0$. Its kernel, 
$\ker \LL$,  
is spanned by a finite number of functions of the form 
\begin{equation}
\label{eq:phi_k}
\phi_k(q,s)=\begin{cases}
\cos(kq) \sin^*(\theta_k s)/\theta_k, & \theta_k\ne 0,\\
\cos(kq) s, & \theta_k = 0,
\end{cases}
\qquad k\in \Z,
\end{equation}
where we have used the notation
\[
\sin^*(\theta_k s) := \begin{cases}\sin(\theta_k s), & \xi k^2+\alpha<0,\\
\sinh(\theta_k s), & \xi k^2+\alpha > 0.
\end{cases}
\]
Define $Z := \{(\eta_\phi,\phi)\colon \phi \in \ker \LL\} \subset \tilde X \subset Y$. Then the range of $\LL$, $\range \LL$, is the orthogonal complement of $Z$ in 
 $Y$ with respect to the inner product
\[
\langle w_1,w_2 \rangle_Y :=\iint_{\hat \Omega} \hat\phi_1 \hat\phi_2\, dq\, ds+\int_{-\pi}^{\pi} \eta_1\eta_2\, dq, \qquad 
w_1,w_2\in Y.
\]
The projection $\Pi_Z$ onto $Z$ along $\range \LL$ is given by 
\[
\Pi_Z w= \sum \frac{\langle w, \tilde w_k\rangle_Y}{\|\tilde w_k\|_Y^2} \tilde w_k,
\]
where the sum ranges over all $\tilde w_k=(\eta_{\phi_k},\phi_k)\in Z$, with $\phi_k$ of the form 
\eqref{eq:phi_k}.
\end{lemma}

The abbreviation $\sin^*$ in Lemma~\ref{lemma:Lproperties} will be used analogously for other trigonometric and hyperbolic functions.


\section{Existence of three-dimensional kernels}\label{sec:threedim}
To ease notation, when $\theta_k=0$ we interpret $\sin^*(\theta_k s)/\theta_k$ as $s$ and $\theta_k \cot^*(\theta_k)$ as $1$. 

\begin{lemma}[Bifurcation condition]
\label{lemma:Bifurcation condition}
Let $\Lambda=(\mu, \alpha,\lambda, \xi)\in \UU$. For $k\in \Z$ we have that $\cos(kq) \sin^*(\theta_k s)/\theta_k\in \ker \LL$ if and only if 
\begin{equation}
\label{eq:bifurcation2}
\theta_k \cot^*(\theta_k) = \frac{1}{\mu^2 \theta_0^2 \sin^2(\lambda)} + \theta_0 \cot(\lambda).
\end{equation}
\end{lemma}
To find nontrivial even functions, we assume that $k>0$. As shown in \cite{eew10} the problem of finding several solutions $k$ for some parameters $\Lambda$ in \eqref{eq:bifurcation2} is a question only of the left-hand side of the same equation. This amounts to finding integer solutions of one or more transcendental equations. Due to the global non-monotonicity of the function $\theta \mapsto \theta \cot(\theta)$ this is quite an intricate question (see \cite[Lemma 4.3]{eew10}). However, if one allows the wave-length parameter $\xi$ to vary, a different approach is possible. Using this, we shall prove that to any two-dimensional kernel in the `trigonometric' regime one may adjoin a third, `hyperbolic', dimension. Such a two-dimensional kernel can be constructed either as in \cite[Lemma 4.3]{eew10} or as in Lemma~\ref{lemma:wavenumbers} below.

Now, let $t = |\alpha|$ and $\xi > 0$, so that $\theta_k := |\xi k^2 - t |^{1/2}$.
Let furthermore $k_2 > k_1$ be positive integers such that
\[
t - \xi k_1^2 > t - \xi k_2^2  >  \pi^2 \quad\text{ and }\quad t - \xi k_j^2 \neq n^2 \pi^2, 
\]
for all $n \in \Z$ and $j=1,2$. We study the (implicit) equation
\[
f(t,\xi)  = 0,
\]
with
\[
f(t,\xi) := \theta_{k_1} \cot(\theta_{k_1}) - \theta_{k_2} \cot(\theta_{k_2}).
\]
When $f(t, \xi)=0$, we set
\[
a(t,\xi) := \theta_{k_1} \cot(\theta_{k_1}) = \theta_{k_2} \cot(\theta_{k_2}).
\]


\begin{proposition}\label{prop:xi}
Suppose that $f(t_0, \xi_0)=0$ and $a(t_0, \xi_0)>1$. 
There is a neighbourhood of $(t_0, \xi_0)$ in which $f(t, \xi)=0$ with $a(t, \xi)>1$ if and only if 
$t=t(\xi)$, where $\xi \mapsto t(\xi)$ is an analytic function defined near $\xi_0$, which satisfies
\[
\frac{dt}{d\xi} = \frac{  \theta_{k_1}^{2} \theta_{k_2}^{2} + t   (a^2-a) }{\xi    (a^2-a)} >\frac{t}{\xi}.
\]
\end{proposition}
\begin{proof}
We have
\begin{align*}
f_t(t,\xi) &= \frac{1}{2} \left( \frac{\cot(\theta_{k_1})}{\theta_{k_1}} - \cot^2(\theta_{k_1}) \right) -  \frac{1}{2}  \left( \frac{\cot(\theta_{k_2})}{\theta_{k_2}} - \cot^2(\theta_{k_2}) \right)\\
&= \frac{a - a^2 }{2 \theta_{k_1}^2 \theta_{k_2}^2} \left( \theta_{k_2}^2  - \theta_{k_1}^2 \right) = \frac{\xi }{2 \theta_{k_1}^{2} \theta_{k_2}^{2}}    (k_2^2 - k_1^2)(a^2-a),
\end{align*}
and
\begin{align*}
f_\xi(t,\xi) &= \frac{-k_1^2}{2}  \left( \frac{\cot(\theta_{k_1})}{\theta_{k_1}} - \cot^2(\theta_{k_1}) -1 \right) +  \frac{k_2^2}{2}  \left( \frac{\cot(\theta_{k_2})}{\theta_{k_2}} - \cot^2(\theta_{k_2}) -1 \right)\\
&= -\frac{1}{2} \left( (k_2^2-k_1^2) + (a-a^2)  \frac{k_1^2 \theta_{k_2}^2 - k_2^2 \theta_{k_1}^2}{\theta_{k_1}^2 \theta_{k_2}^2} \right)\\
&= -\frac{1}{2} (k_2^2 - k_1^2) \left( 1 + \frac{t}{\theta_{k_1}^{2} \theta_{k_2}^{2}}  (a^2-a) \right).
\end{align*}
The proposition now follows from the implicit function theorem and implicit differentiation.
\end{proof}

\begin{proposition}\label{prop:a}
The function $\xi \mapsto t(\xi)$ from Proposition~\ref{prop:xi} induces an analytic diffeomorphism $\xi \mapsto  a(\xi)$ from a bounded interval $(\xi_{min},\xi_{max}) \subset \R_+$ onto $(1,\infty)$ with $\lim_{\xi \searrow \xi_{min}} a(\xi) = \infty$.
\end{proposition}
\begin{proof}
Since
\begin{equation}\label{eq:dxi_theta}
\partial_\xi \theta_{k_1}^2 = \partial_\xi (t(\xi) - \xi k_1^2 ) = t^\prime_\xi - k_1^2 > \frac{t}{\xi} - k_1^2 > \frac{\pi^2}{\xi},	 
\end{equation}
it follows that
\begin{equation}\label{eq:dxi_a}
\begin{aligned}
\partial_\xi a &= \partial_{\theta_{k_1}} \left( \theta_{k_1} \cot(\theta_{k_1}) \right) \partial_{\xi} \theta_{k_1} = \frac{\partial_{\xi} \theta_{k_1}}{\theta_{k_1}} ( a - a^2 - \theta_{k_1}^2 )\\ 
&< -\frac{\pi^2}{2 \xi \theta_{k_1}^2 } ( a^2 - a + \theta_{k_1}^2 ) < -\frac{\pi^2}{2 \xi }. 
\end{aligned}
\end{equation}
This proves that $\xi \mapsto a(\xi)$ defines a local analytic diffeomorphism. The relations~\eqref{eq:dxi_theta} and~\eqref{eq:dxi_a} hold when $\xi > 0$, $t > 0$, $a \in (1, \infty)$ and the entities $\theta_{k_1}$ and $\theta_{k_2}$ satisfy the given assumptions. Consider then a maximally continued parametrization $\xi \mapsto (t(\xi),a(\xi))$. As long as $a \in (1,\infty)$ along such a parametrization the assumptions on  $\theta_{k_1}$, $\theta_{k_2}$ and $t$ cannot be violated. Hence, we only need to determine the set of $\xi$ for which $\xi > 0$ and  $a(\xi) \in (1,\infty)$. This can be deduced from \eqref{eq:dxi_theta}: the differential inequality $\partial_\xi \theta_{k_1}^2 >  \pi^2/\xi$ implies that
\[
\lim_{\xi \nearrow \infty} \theta_{k_1}^2 = \infty, \qquad \lim_{\xi \searrow 0} \theta_{k_1}^2 = -\infty,
\]
the first of which violates $a > 1$, and the second $\theta_{k_1}^2\ge 0$.
Hence there exists $\xi_{min} > 0$ and $\xi_{max} < \infty$ such that Proposition~\ref{prop:a} holds.
\end{proof}

\begin{lemma}[Three-dimensional]\label{lemma:threedim}
For any positive integers $k_2 > k_1$ and positive real numbers $\xi_0, t_0$ with $t_0 - \xi_0 k_1^2 > t_0 - \xi_0 k_2^2 > \pi^2$, 
\[
f(t_0,\xi_0) = 0 \qquad\text{ and }\qquad a(t_0,\xi_0) > 1,
\]
there exist an integer $k_3 > k_2$ and positive real numbers $\xi, t$ such that $\xi k_3^2 - t > 0$ and
\[
\theta_{k_1} \cot(\theta_{k_1}) = \theta_{k_2} \cot(\theta_{k_2}) = \theta_{k_3} \coth(\theta_{k_3}).
\]
The integer $k_3$ may be replaced by any larger integer.
\end{lemma}

\begin{proof}
Choose $k_3 > k_2$ such that $\xi_0 k_3^2 - t_0 > 0$ and 
\[
\sqrt{\xi_0 k_3^2 - t_0} \coth\left(\sqrt{\xi_0 k_3^2 - t_0}\right) > a(t_0,\xi_0).
\]
This is possible in view of that $\lim_{\theta_3 \to \infty}  \theta_3 \coth(\theta_3) = \infty$. Let $\xi \mapsto (t(\xi),a(\xi))$ be the smooth parametrization for which $f(t(\xi), \xi) = 0$ and $a(\xi) \nearrow \infty$ for $\xi \searrow \xi_{min}$. We consider $\xi_{min} < \xi < \xi_0$. Since, for such $\xi$, 
\[
\xi k_3^2 - t(\xi) < \xi k_3^2 < \xi_0 k_3^2,
\]
the function
\[
(\xi_{min}, \xi_0) \ni \xi \mapsto \sqrt{\xi k_3^2 - t(\xi)} \coth\left(\sqrt{\xi k_3^2 - t(\xi)}\right)
\]
is analytic and bounded from above for $\xi k_3^2 - t(\xi) > 0$. In view of that $a(\xi) \nearrow \infty$ as $\xi \searrow \xi_{min}$, it follows that there exists $\xi \in (\xi_{min}, \xi_0)$ with
\[
\sqrt{\xi k_3^2 - t(\xi)} \coth\left(\sqrt{\xi k_3^2 - t(\xi)}\right) = a(\xi).\qedhere
\]
\end{proof}

The construction behind Lemma~\ref{lemma:threedim} does not rule out that the kernel has more than three dimensions. The following result shows that one can construct kernels with exactly three dimensions.

\begin{lemma}\label{lemma:wavenumbers}
For any wavenumbers $k_1, k_2, k_3 \in \Z_{>0}$ with $k_3 > k_2 > k_1$ and
\[
\frac{k_3^2 - k_2^2}{k_3^2 - k_1^2} > \frac{9}{16},
\]
there exist $(\mu,\alpha, \lambda, \xi) \in \UU$ such that \eqref{eq:bifurcation2} holds for $k = k_1, k_2, k_3$ and for no other integer $k \geq k_1$.
\end{lemma}

\begin{proof}
Let $a \in (1,\infty)$. Since $\theta \mapsto \theta \cot{\theta}$ spans $(0,\infty)$ on the intervals $(0,\frac{\pi}{2}) + n\pi$, $n \in \Z_{>0}$, there are $\theta_1=\theta_1(a)\in (2\pi, \frac{5\pi}{2})$ and $\theta_2=\theta_2(a) \in (\pi, \frac{3\pi}{2})$ with
\[
a = \theta_1 \cot{\theta_1} = \theta_2 \cot{\theta_2}.
\]
We are looking for $t=t(a) > 0$ and $\xi=\xi(a) > 0$ such that
\begin{equation*}
\begin{bmatrix}
1 & - k_1^2\\
1 & - k_2^2
\end{bmatrix}
\begin{bmatrix}
t\\
\xi
\end{bmatrix}
=
\begin{bmatrix}
\theta_1^2\\
\theta_2^2
\end{bmatrix}.
\end{equation*}
The unique solution of this linear system is
\begin{equation}\label{eq:alphaxi}
t := \frac{k_2^2 \theta_1^2 - k_1^2 \theta_2^2}{k_2^2 - k_1^2}, \qquad \xi := \frac{\theta_1^2- \theta_2^2}{k_2^2- k_1^2},
\end{equation}
which are both easily seen to be positive. By choosing $\mu$ and $\lambda$ appropriately this yields two integer solutions $k = k_1, k_2$ of~\eqref{eq:bifurcation2}. Since $\theta \cot{\theta} < 1$ for $\theta \in (0,\pi)$, and due to the local monotonicity of $\theta \mapsto \theta \cot{\theta}$ on the intervals $(0,\pi) + n\pi$, $n \in \N$, there can be no other solutions $k \geq k_1$ with $\alpha + \xi k^2 < 0$.

To find a solution $k = k_3$ with $\xi k_3^2 - t > 0$ we consider
\[
\xi(a) k_3^2 -t(a) = \frac{(k_3^2 - k_2^2) \theta_1^2(a) + (k_1^2 - k_3^2) \theta_2^2(a)}{k_2^2 - k_1^2}.
\]
This expression is positive and uniformly bounded away from zero with respect to $a$ whenever $k_1, k_2, k_3$ satisfy the assumptions of the lemma. Since the mapping $a \mapsto (\theta_1(a),\theta_2(a))$ is bounded, uniformly for all $a \in (1,\infty)$, we thus get that
\[
a \mapsto \theta_3(a):=(\xi(a) k_3^2 - t(a))^{1/2}
\]
maps $(1,\infty)$ into a compact interval in $(0,\infty)$. Hence, as $a$ spans $(1,\infty)$ there will be a value of $a$ for which
\[
a = \theta_{k_3}(a) \coth(\theta_{3}(a)).\qedhere
\]
\end{proof}


\section{The Lyapunov--Schmidt reduction and existence of trimodal steady water waves}\label{sec:bifurcation}
Let $\Lambda^*$ denote a quadruple $(\mu^*,\alpha^*, \lambda^*, \xi^*)$ such that \eqref{eq:bifurcation2} holds and suppose that 
 \[\ker \LL(\Lambda^*)=\spn\{\phi^*_1, \dots, \phi^*_n\},\] with 
 $\phi^*_j=\cos(k_j q)\sin^*(\theta_{k_j}s)/\theta_{k_j}$ and 
 $0<k_1<\cdots< k_n$. Let $w^*_j=\TT(\Lambda^*) \phi^*_j$.
From Lemma \ref{lemma:Lproperties} it follows that $Y=Z\oplus \range \LL(\Lambda^*)$. 
As in that lemma, we let $\Pi_Z$ be the corresponding projection onto $Z$ parallel to $\range \LL(\Lambda^*)$. 
This decomposition induces similar decompositions $\tilde X=Z \oplus (\range \LL(\Lambda^*) \cap \tilde X)$ and 
$X=\ker \FF_w(0,\Lambda^*) \oplus X_0$, where $X_0=\RR(\range \LL(\Lambda^*) \cap \tilde X)$ in which 
\[
\RR(\eta, \phi)=\left(\eta, \phi-\frac{s(\psi_0)_s\phi|_{s=1}}{(\psi_0)_s(1)}\right).
\]
Applying the Lyapunov--Schmidt reduction \cite[Thm I.2.3]{MR2004250} we obtain the following lemma.

\begin{lemma}{\cite{MR2004250}}
\label{lemma:Lyapunov-Schmidt}
There exist open neighborhoods $\NN$ of $0$ in $\ker \FF_w(0,\Lambda^*)$, $\MM$ of $0$ in $X_0$ and $\UU^* $ of $\Lambda^*$ in $\R^4$, and a function
 $\psi \in C^\infty(\NN \times \UU^* , \MM)$, 
such that
\[
\FF(w,\Lambda)=0 \quad\text{ for }\quad w\in \NN+\MM,\quad \Lambda\in \UU^*,
\]
if and only if $w=w^*+\psi(w^*, \Lambda)$ and $w^*=t_1 w^*_1+\cdots+t_n w_n^*\in \NN$ solves the 
finite-dimensional problem
\begin{equation}
\label{eq:Lyapunov-Schmidt}
\Phi(t, \Lambda)=0 \quad\text{ for }\quad t \in \VV, \quad \Lambda \in \UU^*,
\end{equation}
in which
\[
\Phi(t, \Lambda):=\Pi_Z \FF(w^*+\psi(w^*, \Lambda), \Lambda), \quad 
\]
and $\VV := \{t \in \R^n \colon t_1 w^*_1+\cdots+t_n w_n^*\in \NN\}$. The function $\psi$ has the properties $\psi(0,\Lambda)=0$ and $D_w \psi(0, \Lambda)=0$.
\end{lemma}

\subsubsection*{Bifurcation from a three-dimensional kernel}

\begin{theorem}[Three-dimensional bifurcation]\label{thm:threedim}
Suppose that
\[
\dim \ker \Diff_w \FF(0,\Lambda^*)=3,
\]
and that
\begin{equation}\label{eq:a}
a:= \theta_{k_1}\cot^*(\theta_{k_1})= \theta_{k_2}\cot^*(\theta_{k_2}) = \theta_{k_3}\coth(\theta_{k_3}) > 1.
\end{equation}
Assume also that the integers $k_1 < k_2 < k_3$ are positive, and let $t := (t_1, t_2, t_3)$. Then there exists for every $\delta \in (0,1)$ a smooth family of small-amplitude nontrivial solutions 
\begin{equation}\label{eq:mainset}
\SS := \{(\overline w(t),\overline \mu(t), \overline \alpha(t), \overline \xi (t)) \colon 0<|t|<\ve, |t_1t_2 t_3|>\delta |t|^3 \}
\end{equation}
of $\FF(w,\mu,\alpha, \lambda^*, \xi) = 0$ in $\OO \times \R^3$, passing 
through $(\overline w(0), \overline \mu(0), \overline \alpha(0), \overline \xi (0))=(0, \mu^*, \alpha^*, \xi^*)$ with 
\[
\overline w(t)=t_1 w_1^*+t_2 w_2^* + t_3 w_3^* +O(|t|^2)
\]
\end{theorem}

\begin{proof}[Proof of Theorem \ref{thm:threedim}]
The first part of the proof is analogous to that of \cite[Theorem 4.8]{eew10} (see also \cite{BaldiToland10}). The second part involves calculating the determinant of $3 \times 3$-matrix, the entries of which depend transcendentally on the parameters of the problem. This can be done via thorough, but rudimentary, investigation.

\subsection*{Part I. Reduction to a $3\times3$ function-valued matrix.}
Define $\tilde w^*_j := ( \eta_{\phi_j^*},\phi_j^*) \in \tilde X$, $j=1,2,3$, and recall that $Z=\spn\{\tilde w^*_1, \tilde w^*_2, \tilde w^*_3\}$. If $\Pi_{j} \Phi=\Phi_j \tilde w_j^*$, where $\Pi_{j}$ denotes the projection onto $\spn \{\tilde w_j^*\}$, equation \eqref{eq:Lyapunov-Schmidt} takes the form
\begin{equation}
\label{eq:reduced system}
\begin{aligned}
\Phi_j(t, \Lambda)&=0,\qquad j = 1,2,3.
\end{aligned}
\end{equation}
This is a system of three equations with seven unknowns, and we note that it has the trivial solution $(0,\Lambda)$ for all $\Lambda\in \UU^*$. 
\medskip

We introduce polar coordinates by writing $t= r\hat t$ with $|\hat t|=1$.
Then
\[
\Phi_j(r \hat t, \Lambda)=r \Psi_j(r \hat t, \Lambda), 
\]
where
\[
\Psi_j(t, \Lambda):=\int_0^1 \hat t \cdot \nabla_t \Phi_j(z r\hat t, \Lambda) \, dz, \qquad j = 1,2,3,
\]
since $\Phi_j(0, \Lambda)=0$.
\eqref{eq:reduced system} is equivalent to
\begin{equation}
\label{eq:reduced system 2}
\begin{aligned}
r\Psi_j(t, \Lambda)&=0,\qquad j =1,2,3.
\end{aligned}
\end{equation}
Since $\pa_{t_j} \Phi(0,\Lambda^*)=\Pi_Z \Diff_w \FF(0, \Lambda^*) w_j^*=0$,
we have that 
\begin{equation}\label{eq:psiphiat0}
\Psi_j(0, \Lambda^*)= \hat t \cdot \nabla_{t} \Phi_j(0, \Lambda^*)=0, \qquad j = 1,2,3.
\end{equation}
We therefore apply the implicit function theorem to $\Psi$ at the point $(0, \Lambda^*)$, by proving that the matrix 
\begin{equation}\label{eq:matrixM}
M:=
\begin{bmatrix}
\pa_\mu \Psi_1(0,\Lambda^*) & \pa_\mu \Psi_2(0,\Lambda^*) & \pa_\mu \Psi_3(0,\Lambda^*)\\
\pa_\alpha \Psi_1(0,\Lambda^*) &  \pa_\alpha \Psi_2(0,\Lambda^*) & \pa_\alpha \Psi_3(0,\Lambda^*)\\
\pa_\xi \Psi_1(0,\Lambda^*) &  \pa_\xi \Psi_2(0,\Lambda^*) & \pa_\xi \Psi_3(0,\Lambda^*)\\
\end{bmatrix}
\end{equation}
is invertible. We have that
\begin{align*}
\Pi_Z \Diff^2_{w\mu} \FF(0, \Lambda^*) w_j^*&=
 \frac{\langle \Diff^2_{w\mu} \FF(0,\Lambda^*) w_j^*, \tilde w_j^*\rangle_Y}{\| \tilde w_j^*\|_Y^2}  \tilde w_j^*,
\end{align*}
since $\Diff^2_{w\mu} \FF(0,\Lambda^*) w_j^*$ is orthogonal to $\tilde w_i^*$ for $i \neq j$.
In view of that $\pa_\mu \Psi_j(0,\Lambda^*) = \sum_{i=1}^3 \hat t_i \pa_{t_i}\pa_\mu \Phi_j(0, \Lambda^*)$
and $\pa_{t_i}\pa_\mu \Phi(0,\Lambda^*)=\Pi_Z \Diff^2_{w\mu} \FF(0, \Lambda^*) w_i^*$ we thus have that
\begin{equation}\label{eq:tjfallingout}
\pa_\mu \Psi_j(0,\Lambda^*) =\sum_{i=1}^3 \hat t_i \pa_{t_i}\pa_\mu \Phi_j(0, \Lambda^*)=\hat t_j\frac{\langle \Diff^2_{w\mu} \FF(0,\Lambda^*) w_j^*,\tilde w_j^*\rangle_Y}{\| \tilde w_j^*\|_Y^2} .
\end{equation}
Recall that
$\Diff_w \FF(0,\Lambda)=\LL(\Lambda)\TT^{-1}(\Lambda^\prime)$. Thus, 
\[
\Diff^2_{w \mu} \FF(0,\Lambda^*)w_j^*=\Diff_\mu \mathcal L(\Lambda^*)\phi^*+\LL(\Lambda^*)\Diff_\mu \TT^{-1}(\Lambda^*) w_j^*.
\]
Since $\LL = \Diff_w \FF(0) \circ \TT$ the second term on the right-hand side belongs to $\range \Diff_w \FF(0,\Lambda^*)$, and we find that
$\langle \Diff^2_{w\mu} \FF(0,\Lambda^*) w_j^*, \tilde w_j^*\rangle_Y=
\langle \Diff_\mu \LL(\Lambda^*)\phi_j^*, \tilde w_j^*\rangle_Y$. 
Thus
\[
\pa_\mu \Psi_j(0,\Lambda^*)=\frac{\langle \Diff_\mu \LL(\Lambda^*) \phi_j^*, \tilde w_j^*\rangle_Y}{\| \tilde w_j^*\|_Y^2}, \qquad j = 1,2,3. 
\]
Similar arguments hold for $\pa_\mu \Psi_j(0,\Lambda^*)$ and $\pa_\alpha \Psi_j(0,\Lambda^*)$,  and we find that
\begin{equation}\label{eq:determinantM}
\begin{aligned}
\det M =
C\hat t_1 \hat t_2\hat t_3 \det 
\begin{bmatrix}
\langle \Diff_{\mu} \LL(\Lambda^*) \phi^*_1 ,\tilde w^*_1 \rangle_Y & 
\langle \Diff_{\mu} \LL(\Lambda^*) \phi^*_2 ,\tilde w^*_2  \rangle_Y &
\langle \Diff_{\mu} \LL(\Lambda^*) \phi^*_3 ,\tilde w^*_3  \rangle_Y\\
\langle \Diff_{\alpha} \LL(\Lambda^*) \phi^*_1,\tilde w^*_1 \rangle_Y & 
\langle \Diff_{\alpha} \LL (\Lambda^*) \phi^*_2,\tilde w^*_2 \rangle_Y &
\langle \Diff_{\alpha} \LL (\Lambda^*) \phi^*_3,\tilde w^*_3 \rangle_Y \\
\langle \Diff_{\xi} \LL(\Lambda^*) \phi^*_1,\tilde w^*_1 \rangle_Y & 
\langle \Diff_{\xi} \LL (\Lambda^*) \phi^*_2,\tilde w^*_2 \rangle_Y &
\langle \Diff_{\xi} \LL (\Lambda^*) \phi^*_3,\tilde w^*_3 \rangle_Y \\
\end{bmatrix},
\end{aligned}
\end{equation}
where
$C=\| \tilde w_1^*\|_Y^{-2} \| \tilde w_2^*\|_Y^{-2} \| \tilde w_3^*\|_Y^{-2}\ne 0$.

\subsection*{Part II. Determining $\det(M)$.}
Let $j \in \{1,2,3\}$, and
\[
\pm^* := \sgn(\xi k_j^2+\alpha).
\]
One has (see \cite{eew10})
 \begin{align*}
\left\langle \Diff_{\mu} \LL (\Lambda^*) \phi^*_j ,\tilde w^*_j \right\rangle_Y &= A \left(\frac{\sin^*(\theta_{k_j})}{\theta_{k_j}}\right)^2, 
 \end{align*} 
and
\begin{align*}
\left\langle \Diff_{\alpha} \LL(\Lambda^*) \phi^*_j ,\tilde w^*_j \right\rangle_Y=B\left(\frac{\sin^*(\theta_{k_j})}{\theta_{k_j}}\right)^2+f(k_j),
\end{align*}
where $A$ and $B$ are constants, $A$ is non-zero, and 
\begin{equation}\label{eq:f-function}
f(k_j):=
\pm^* \, \frac\pi2  \frac{\theta_{k_j}-\cos^*(\theta_{k_j})\sin^*(\theta_{k_j})}{\theta_{k_j}^3},  
 \end{equation}
which is naturally extended to a continuous function of $\theta_{k_j}$.
It follows from~\eqref{eq:Lexpression} that $\Diff_\xi \LL(\Lambda) \phi = (0, \partial_q^2 \phi)$, whence
\[
\Diff_\xi \LL(\Lambda^*) \phi^*_j = (0, -k_j^2 \phi^*_j).
\]
This is $k_j^2$ times the second component of $\Diff_{\alpha} \LL(\Lambda^*) \phi^*_j$, and  straightforward integration shows that
\[
\left\langle \Diff_{\xi} \LL(\Lambda^*) \phi^*_j ,\tilde w^*_j \right\rangle_Y= k_j^2 f(k_j).
\] 
We are thus left with calculating the determinant
\allowdisplaybreaks
\begin{equation}\label{eq:determinant}
\begin{aligned}
&\det \begin{pmatrix}
\frac{A (\sin^*(\theta_{k_1}))^2}{\theta_{k_1}^2} & 
\frac{A (\sin^*(\theta_{k_2}))^2}{\theta_{k_2}^2} &
\frac{A (\sin^*(\theta_{k_3}))^2}{\theta_{k_3}^2}\\
\frac{B (\sin^*(\theta_{k_1}))^2}{\theta_{k_1}^2}+f(k_1) & 
\frac{B (\sin^*(\theta_{k_2}))^2}{\theta_{k_2}^2}+f(k_2) &
\frac{B (\sin^*(\theta_{k_3}))^2}{\theta_{k_3}^2}+f(k_3)\\
\qquad k_1^2 f(k_1) & \qquad k_2^2 f(k_2) & \qquad k_3^2 f(k_3)  
\end{pmatrix}\\
&= A \Bigg( f(k_1) f(k_2) (k_2^2 - k_1^2)\, \frac{(\sin^*(\theta_{k_3}))^2}{\theta_{k_3}^2}  + f(k_1) f(k_3) (k_1^2 - k_3^2)\, \frac{(\sin^*(\theta_{k_2}))^2}{\theta_{k_2}^2}\\
&\qquad\quad + f(k_2) f(k_3) (k_3^2 - k_2^2)\, \frac{(\sin^*(\theta_{k_1}))^2}{\theta_{k_1}^2} \Bigg)\\
&= A \prod_{j=1}^3 f(k_j) \Bigg(  (k_2^2 - k_1^2)\, \frac{(\sin^*(\theta_{k_3}))^2}{\theta_{k_3}^2 f(k_3)}  +  (k_1^2 - k_3^2)\, \frac{(\sin^*(\theta_{k_2}))^2}{ \theta_{k_2}^2 f(k_2)}\\
&\qquad\quad + (k_3^2 - k_2^2)\, \frac{(\sin^*(\theta_{k_1}))^2}{ \theta_{k_1}^2 f(k_1)} \Bigg).
\end{aligned}
\end{equation}
The function $\theta \mapsto f(\theta)$ defined by \eqref{eq:f-function} is everywhere negative, whence we are left with investigating the expression
\begin{align*}
\tilde f(\theta) &:= \frac\pi2 \frac{(\sin^*(\theta))^2}{\theta^2 f(\theta)}\\ 
&= \pm^* \frac{(\sin^*(\theta))^2}{\theta^2} \left( \frac{\theta^3}{\theta - \cos^*(\theta) \sin^*(\theta)} \right)\\
&= \pm^* \frac{a (\sin^*(\theta))^2}{a - \frac{a \cos^*(\theta) \sin^*(\theta)}{\theta}}\\
&=  \frac{a  (\sin^*(\theta))^2}{\pm^*\,(a -1) - (\sin^*(\theta))^2}.
\end{align*}
Note that the denominator is strictly negative: for $\xi k_j^2+\alpha < 0$ since $a > 1$, and for $\xi k_j^2+\alpha > 0$ since  $\theta \coth{\theta} - 1 - \sinh^2(\theta) < 0$ for all $\theta \neq 0$. 
Let $x:= (\sin^*(\theta))^2$. For $a > 1$ the function $x \mapsto \frac{ax}{a-1-x}$ is strictly increasing with limit $-a$ as $x \to \infty$, so we immediately get that
\begin{equation}\label{eq:tildef3}
\tilde f(\theta_{k_3}) < -1.
\end{equation}
When $\xi k_j^2+\alpha < 0$ the function $[0,1] \ni x \mapsto \frac{ax}{1-a-x}$ is strictly decreasing with image $[-1,0]$. Since
\[
\theta_{k_1} \cot(\theta_{k_1}) = \theta_{k_2} \cot(\theta_{k_2}) < \theta_{k_1} \cot(\theta_{k_2}),
\]
with the left-hand side positive, we may square both sides to obtain that
\[
\sin^2(\theta_{k_1}) >  \sin^2(\theta_{k_2}). 
\]
Taking \eqref{eq:tildef3} into consideration, we thus conclude that
\[
\tilde f(\theta_{k_3}) < \tilde f(\theta_{k_1}) < \tilde f(\theta_{k_2}) < 0.
\]
Returning to \eqref{eq:determinant}, this shows that the determinant of $M$ is non-zero:
\begin{align*}
&(k_2^2 - k_1^2) \tilde f(\theta_{k_3})  +  (k_1^2 - k_3^2)\tilde f(\theta_{k_2}) + (k_3^2 - k_2^2)\tilde f(\theta_{k_1})\\
&= (k_2^2 - k_1^2) ( \tilde f(\theta_{k_3}) - \tilde f(\theta_{k_2}) )  + (k_3^2 - k_2^2) ( \tilde f(\theta_{k_1}) - \tilde f (\theta_{k_2})) < 0.
\end{align*} 
The condition $|t_1 t_2 t_3|>\delta |t|^3$ implies that $|\hat t_1 \hat t_2 \hat t_3|>\delta$, so that the determinant is uniformly bounded from below. This guarantees that the interval $0<|t|<\ve$ can be chosen uniformly.
\end{proof}


\section{The structure of the solution set}\label{sec:solutions}

Theorem~\ref{thm:threedim} does not give the full local solution set near the bifurcation point. In particular, since the solutions are bounded away from the coordinate planes in \((t_1,t_2,t_3)\)-space, any solutions obtained by restricting the period and using bifurcation with a one- or two-dimensional kernel are excluded. In this section we present a method which gives a more complete picture of the solution set by taking into account the number theoretic properties of the integers $k_1, k_2, k_3$. 
While the method is not guaranteed to find the full local solution set, the solutions obtained by lower-dimensional bifurcation are included.

For the purpose of the following discussion it is conveniant to ignore the order of the numbers $k_1, k_2, k_3$. We therefore relabel them as $m_1, m_2, m_3$, with no particular order, but assuming that  $\gcd(m_1, m_2, m_3)=1$.
If $m=\gcd(m_1, m_2, m_3)>1$, then by working in the subspace $X^{(m)}$ of $2\pi/m$-periodic functions we can reduce the problem to the previous case. We therefore have the following four different cases.

\begin{align}
&\gcd(m_1, m_2)>1,\quad& \gcd(m_1, m_3)>1,\quad& \gcd(m_2, m_3)>1.\label{case:i}\tag{i}\\
&\gcd(m_1, m_2)=1, \quad&\gcd(m_1, m_3)>1, \quad &\gcd(m_2, m_3)>1.\label{case:ii}\tag{ii}\\
&\gcd(m_1, m_2)=1, \quad&\gcd(m_1, m_3)=1, \quad&\gcd(m_2, m_3)>1.\label{case:iii}\tag{iii}\\
&\gcd(m_1, m_2)=1, \quad&\gcd(m_1, m_3)=1,\quad&\gcd(m_2, m_3)=1.\label{case:iv}\tag{iv}
\end{align}

Each of the above cases have subcases determined by whether the different $m_j$ divide each other or not. In what follows, all cases will be presented in the form of the reduced equations they give rise to, as well as the corresponding solution set in \((t_1,t_2, t_3)\)-space (cf. Theorem~\ref{thm:threedim} and Equation~\eqref{eq:mainset}). We recall that this is a subset of the open ball
\begin{equation}\label{eq:ballofsolutions}
\BB := \{(\overline w(t),\overline \mu(t), \overline \alpha(t), \overline \xi (t)) \colon 0<|t|<\ve \}
\end{equation}
of possible solutions. The method is illustrated with more details for the first cases, whereas the last and analogue cases are presented in a shorter manner.
\medskip

\subsection{Case \eqref{case:i}: an open ball of solutions.} \label{subsec:anopenball}
In this case no $m_i$ can divide another 
$m_j$. Indeed, assume e.g.~that $m_1 \mid m_2$. Then $\gcd(m_1, m_3) \mid m_2$ as well, so that $\gcd(m_1, m_2, m_3)=\gcd(m_1, m_3)>1$, yielding a contradiction. A numerical example is given by $(6, 10, 15)$.

The relations between the $m_j$ imply that
\begin{equation}\label{eq:cases_i}
\begin{aligned}
&\Phi_1(0,t_2,t_3)=0,\\
&\Phi_2(t_1, 0, t_3)=0,\\
&\Phi_3(t_1, t_2, 0)=0,
\end{aligned}
\end{equation}
where $\Lambda$ has been supressed for convenience. Then \eqref{eq:cases_i} is equivalent to that
\begin{equation}\label{eq:tjphij}
\Phi_j(t, \Lambda) = \int_0^1 \frac{d}{dz} \Phi_j(zt_j; t_{i}|_{i\neq j}, \Lambda)\,dz =  t_j \Psi_j(t, \Lambda) = 0,
\end{equation}
with
\[
\Psi_j(t, \Lambda):=\int_0^1 \Diff_{t_j} \Phi_j(zt_j; t_{i}|_{i\neq j}, \Lambda)\,dz,
\]
both for $j=1,2,3$. At the point $(0,\Lambda^*)$, the relation $\Psi_j(0, \Lambda^*) = \Diff_{t_j} \Phi_j(0, \Lambda^*)$ enables us to apply the implicit function theorem to \eqref{eq:matrixM} without any $t$-dependent coefficients appearing before the matrix in \eqref{eq:determinantM} (cf. \eqref{eq:psiphiat0}); the result is a full three-dimensional ball of solutions, 
\[
\SS_{(i)} = \BB,
\] 
as given in \eqref{eq:ballofsolutions}.
\medskip

\begin{figure}
\includegraphics[width=0.2\linewidth, height=0.20\linewidth]{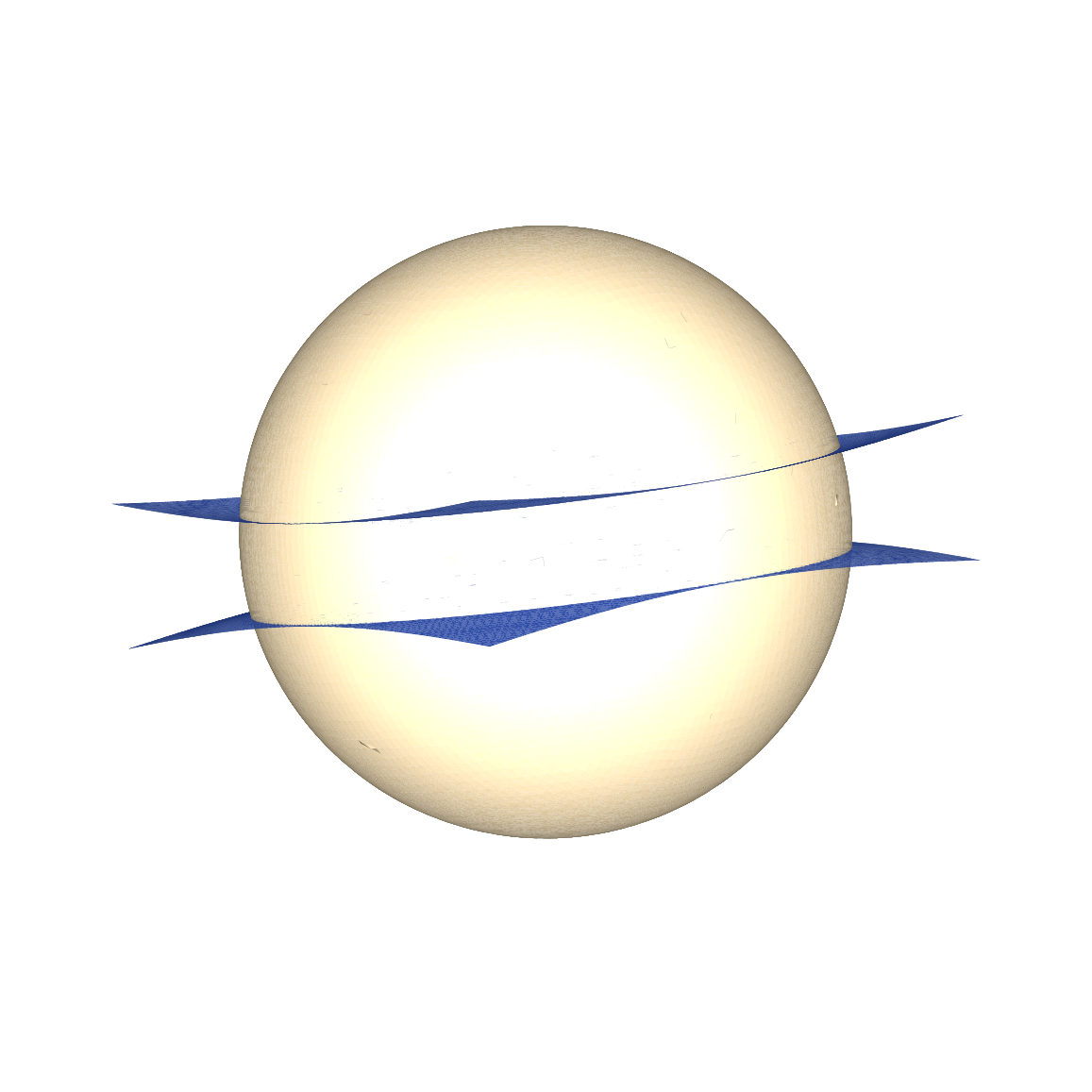}\qquad
\includegraphics[width=0.2\linewidth, height=0.20\linewidth]{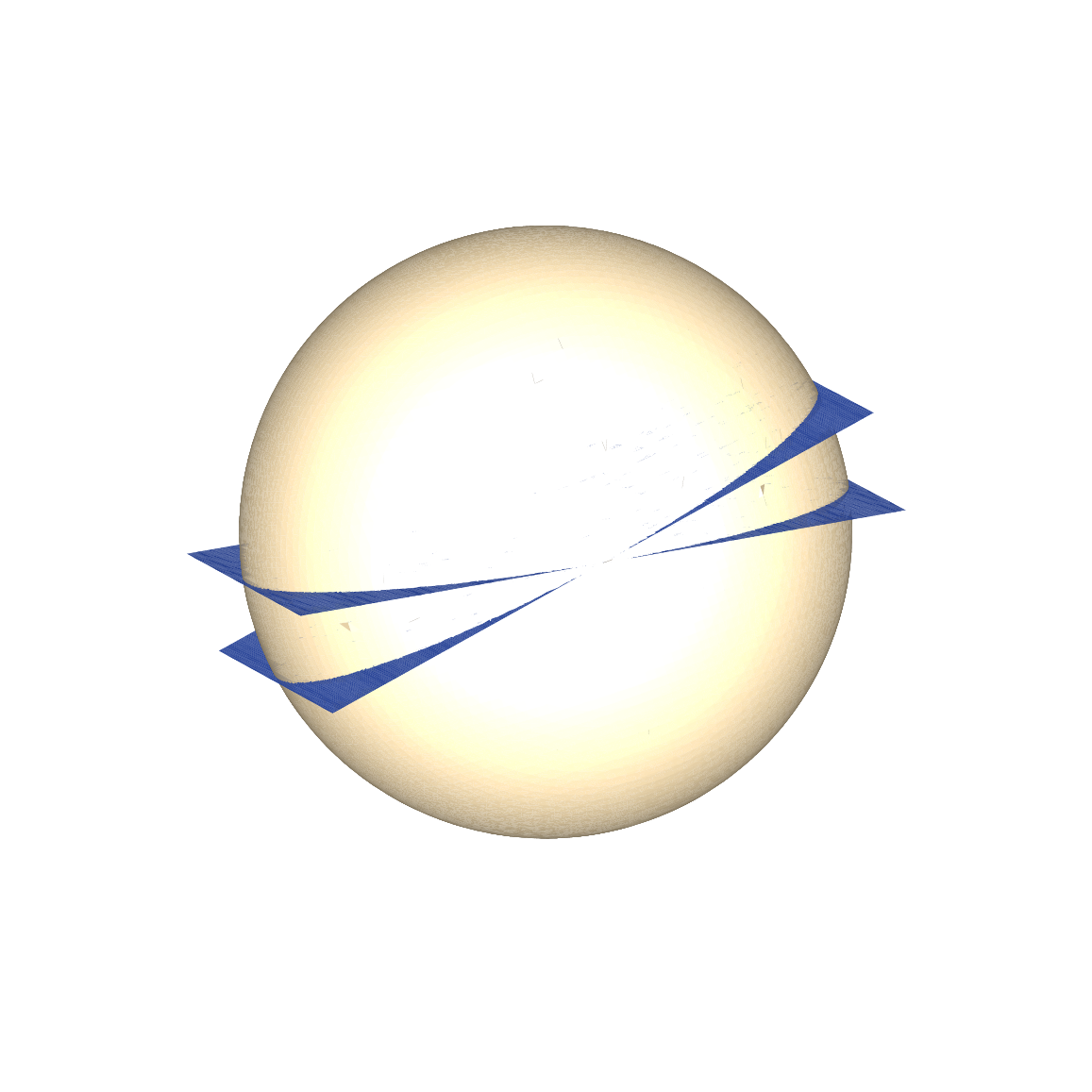}\qquad
\includegraphics[width=0.2\linewidth, height=0.20\linewidth]{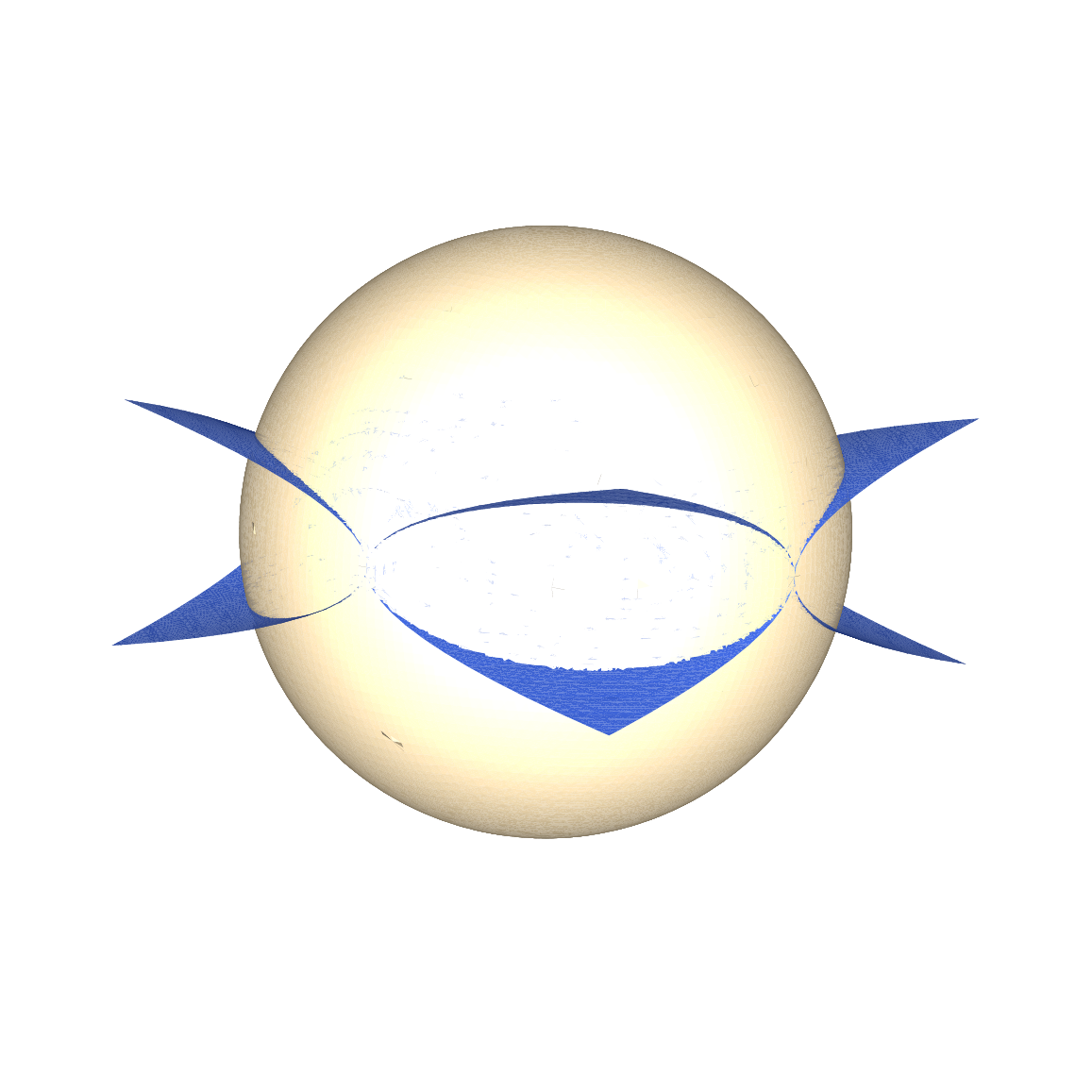}
\caption{Left: The cases (ii)\,a, (ii)\,b and (ii)\,c, in order from left to right. The illustrations show the qualitative intersection of the solution set \(\SS_{(ii)_{\alpha}}\), \(\alpha=a,b,c\), with the ball \(\BB\) of radius \(\varepsilon > 0\) in \((t_1,t_2,t_3)\)-space.} 
\label{fig:caseii}
\end{figure}

\subsection{Case \eqref{case:ii}}
Using the same argument as above, we notice that $m_3 \nmid m_1, m_2$.
Also, $m_1$ cannot divide $m_2$, or contrariwise, since this would give $m_1=1$ and therefore 
would contradict $\gcd(m_1, m_3)>1$.
Up to relabeling, we therefore obtain three alternatives:
\begin{enumerate}
\item[(a)] $m_1 \mid m_3$, $m_2 \mid m_3$. An example of this is $(2, 3, 6)$. One obtains
\begin{align*}
&\Phi_1(0,t_2, t_3)=0, \\
&\Phi_2(t_1, 0,t_3)=0, \\
& \Phi_3(0, 0, 0)=0.
\end{align*}
Here \eqref{eq:tjphij} holds for \(j=1,2\), whereas for \(\Phi_3\) we use spherical coordinates \(t = r\hat t\), writing
\[
\Phi_3(t, \Lambda) = r \Psi_3(t,\Lambda) \quad\text{ with }\quad \Psi_3(0,\Lambda^*) = \hat t \cdot \nabla_t \Phi_3(0, \Lambda^*).
\] 
An iteration of the proof on page~\pageref{eq:tjfallingout}, see especially \eqref{eq:tjfallingout}, then yields that the system
\begin{align*}
t_1 \Psi_1(t,\Lambda) &= 0,\\
t_2 \Psi_2(t,\Lambda) &= 0,\\
r \Psi_3(t,\Lambda) &= 0,
\end{align*}
can be solved using the implicit function theorem whenever $|\hat t_3| \geq \delta$, i.e., whenever $|t_3| \geq \delta |t|$. Hence, for every $\delta \ll 1$, the restriction
\[
\SS_{(ii)_a} :  |t_3| \geq \delta |t|
\] 
of the ball \(\BB\) describes a smooth family of small-amplitude nontrivial solutions. The solutions that may be found by one- and two-dimensional bifurcation by setting the different $t_j$'s to zero, namely
\begin{align*}
&t_1 = 0, \quad 0 < (t_2^2 + t_3^2)^{1/2} < \varepsilon, \quad |t_3| \geq \delta |t_2| \\
&t_2 = 0,  \quad 0 < (t_1^2 + t_3^2)^{1/2} < \varepsilon,  \quad |t_3| \geq \delta |t_1|,
\end{align*}
are included in this solution set; and for $t_3=0$ with $t_1^2 + t_2^2 \neq 0$ no solutions are found.\\

\item[(b)] $m_1 \mid m_3$, $m_2 \nmid m_3$. An example of this is $(2, 9, 12)$. One finds
\begin{align*}
&\Phi_1(0,t_2, t_3)=0, \\
&\Phi_2(t_1, 0,t_3)=0, \\
& \Phi_3(0, t_2, 0)=0.
\end{align*}
Again \eqref{eq:tjphij} holds for \(j=1,2\), whereas for \(\Phi_3\) we use \emph{cylindrical coordinates} \(t = (r_{1,3}\hat t_1, t_2, r_{1,3} \hat t_3)\), with \(r_{i,j} := (t_i^2 + t_j^2)^{1/2}\). This yields
\[
\begin{aligned}
\Phi_3(t, \Lambda) &= \int_0^1 \frac{d}{dz} \Phi_3(z r_{1,3} (\hat t_1,\hat t_3);t_2, \Lambda)\,dz\\ &=  r_{1,3} \int_0^1 (\hat t_1, \hat t_3) \cdot \nabla_{(t_1, t_3)} \Phi_3(zr_{1,3}(\hat t_1,\hat t_3);t_2, \Lambda)\,dz\\ &= r_{1,3} \Psi_3(t,\Lambda),
\end{aligned}
\]
with
\[ 
\Psi_3(0,\Lambda^*) = (\hat t_1, \hat t_3) \cdot \nabla_{(t_1, t_3)} \Phi_3(0, \Lambda^*),
\] 
and the system of equations now becomes
\begin{align*}
t_1 \Psi_1(t,\Lambda) &= 0,\\
t_2 \Psi_2(t,\Lambda) &= 0,\\
r_{1,3} \Psi_3(t,\Lambda) &= 0.
\end{align*}
This can be solved whenever $|\hat t_3| \geq \delta$, i.e., for $|t_3| \geq \delta (t_1^2 + t_3^2)^{1/2}$. Hence, for any $\delta \ll 1$, we obtain the family of solutions given by the restriction
\[
\SS_{(ii)_b} : |t_3| \geq \delta |t_1|
\] 
to \(\BB\). As in the case $(ii)_a$, any 'lower-dimensional' solutions,
\begin{align*}
&t_1 = 0, \quad 0 < (t_2^2 + t_3^2)^{1/2} < \varepsilon,\\
&t_2 = 0,  \quad 0 < (t_1^2 + t_3^2)^{1/2} < \varepsilon,  \quad |t_3| \geq \delta |t_1|,\\
&t_1 = t_3 = 0, \quad 0 < |t_2| < \varepsilon;
\end{align*}
are included in the above solution set; and for $t_3=0$ with $t_1 \neq 0$ the implicit function theorem is inconclusive.
\medskip

\item[(c)] $m_1 \nmid m_3$, $m_2 \nmid m_3$. An example is $(4, 9, 30)$.
\begin{align*}
&\Phi_1(0,t_2, t_3)=0, \\
&\Phi_2(t_1, 0,t_3)=0, \\
& \Phi_3(t_1, 0,0)=\Phi_3(0, t_2, 0)=0.
\end{align*}
The difference with respect to case $(iii)_b$ is that we may express \(\Phi_3\) using different cylindrical coordinates as either \(t = (r_{1,3}\hat t_1, t_2, r_{1,3} \hat t_3)\) or \(t = (t_1, r_{2,3} \hat t_2, r_{2,3} \hat t_3)\),  where $r_{i,j} = (t_i^2 + t_j^2)^{1/2}$. Thus, the original system reduces to
\begin{align*}
t_j \Psi_j(t,\Lambda) &= 0,\\
r_{j,3} \Psi_3(t,\Lambda) &= 0,
\end{align*}
for \(j=1,2\). This can be solved whenever $|t_3| \geq \delta (t_j^2 + t_3^2)^{1/2}$ for \emph{either} $j=1$ or $j=2$, and we obtain
\[
\SS_{(ii)_c} :|t_3| \geq \delta \min(|t_1|, |t_2|).
\] 
The solutions obtained from lower-dimensional bifurcation,
\begin{align*}
&t_i = 0,  \quad 0 < (t_j^2 + t_3^2)^{1/2} < \varepsilon,  \quad |t_3| \geq \delta |t_j|,\\
&t_i = t_3 = 0, \quad 0 < |t_j| < \varepsilon,
\end{align*}
for $i,j = 1,2$, $i \neq j$, are all included in the larger solution set; and for $t_3=0$ with $t_1 t_2 \neq 0$ no solutions are found.
\end{enumerate}
\medskip

\begin{figure}
\includegraphics[width=0.2\linewidth, height=0.20\linewidth]{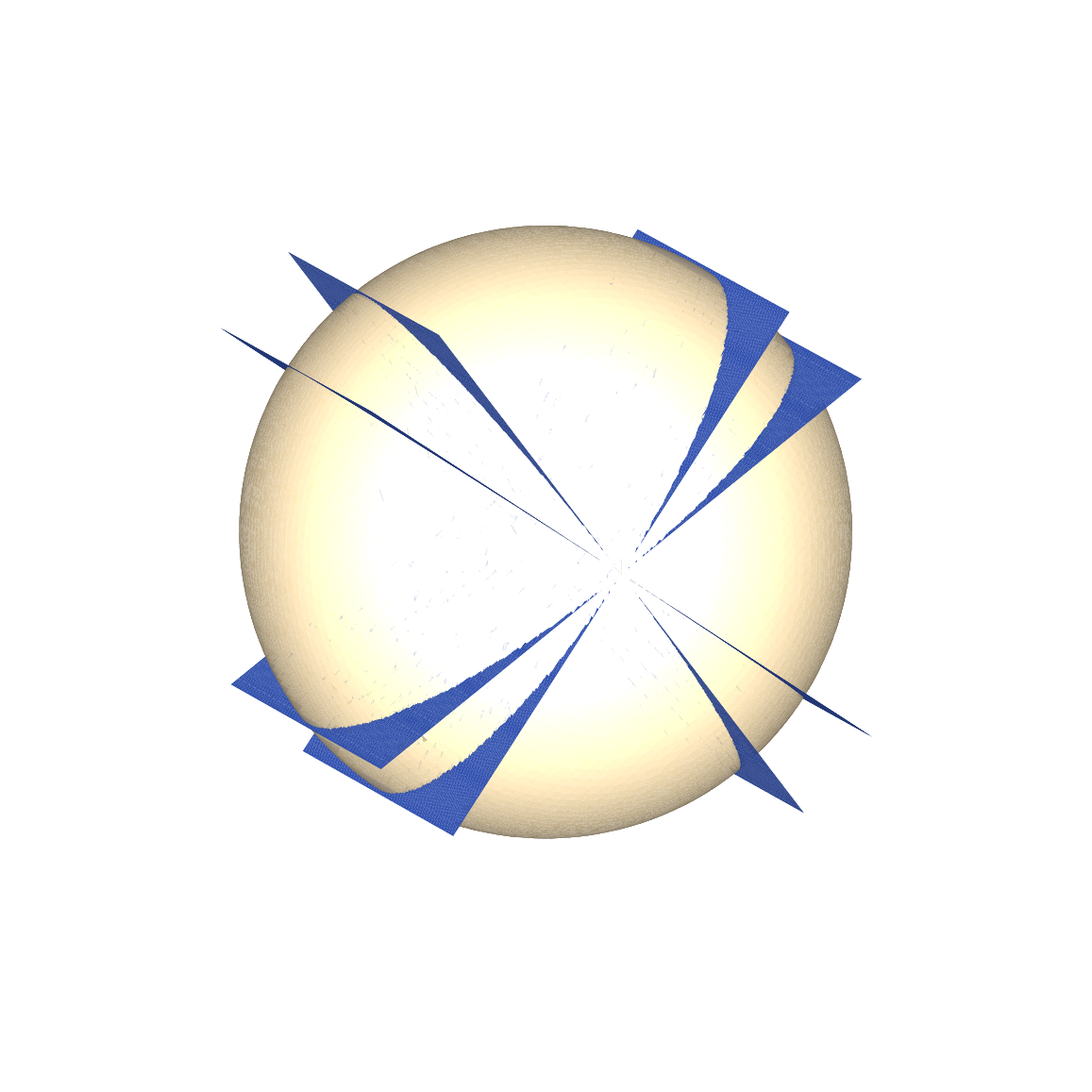}\qquad
\includegraphics[width=0.2\linewidth, height=0.20\linewidth]{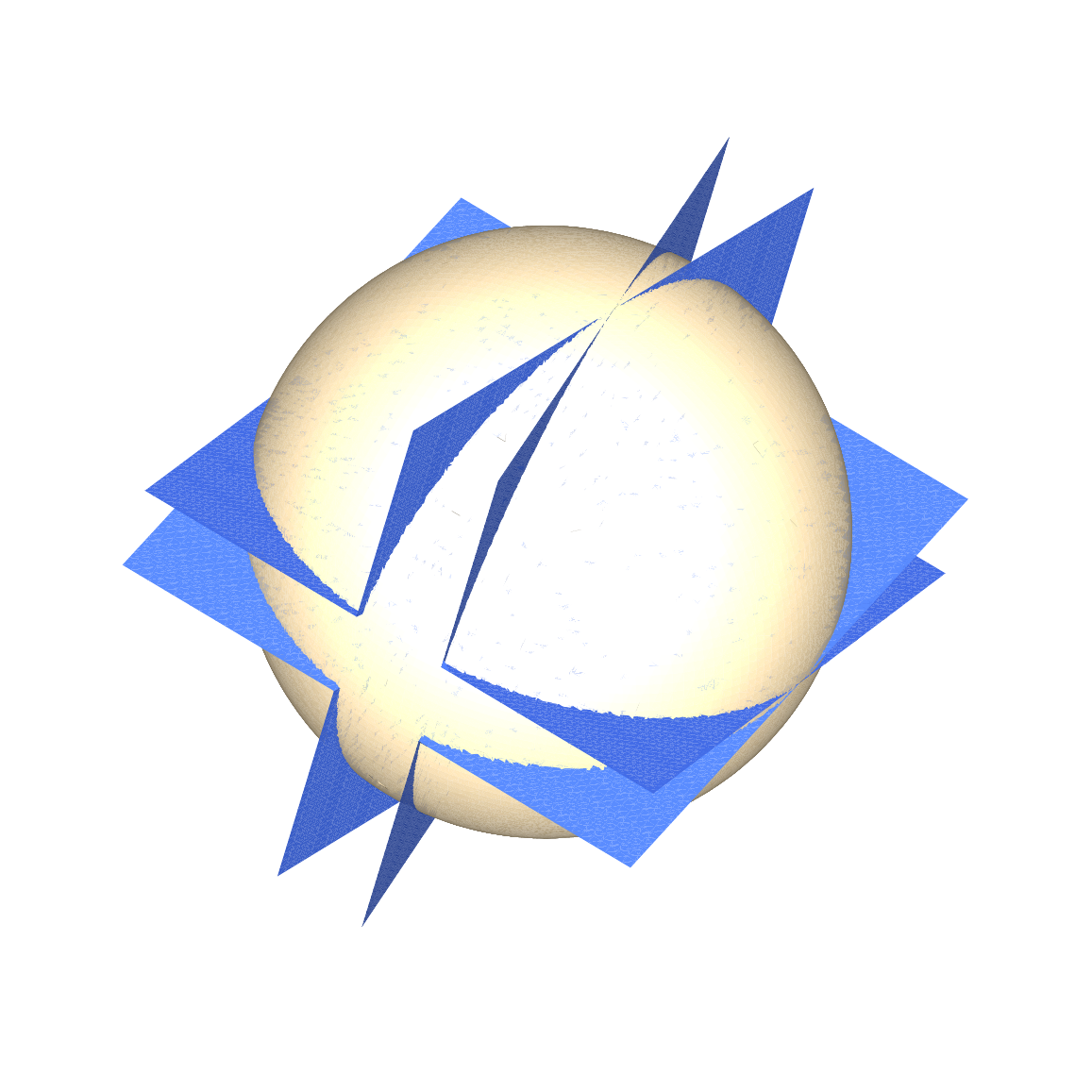}\qquad
\includegraphics[width=0.2\linewidth, height=0.20\linewidth]{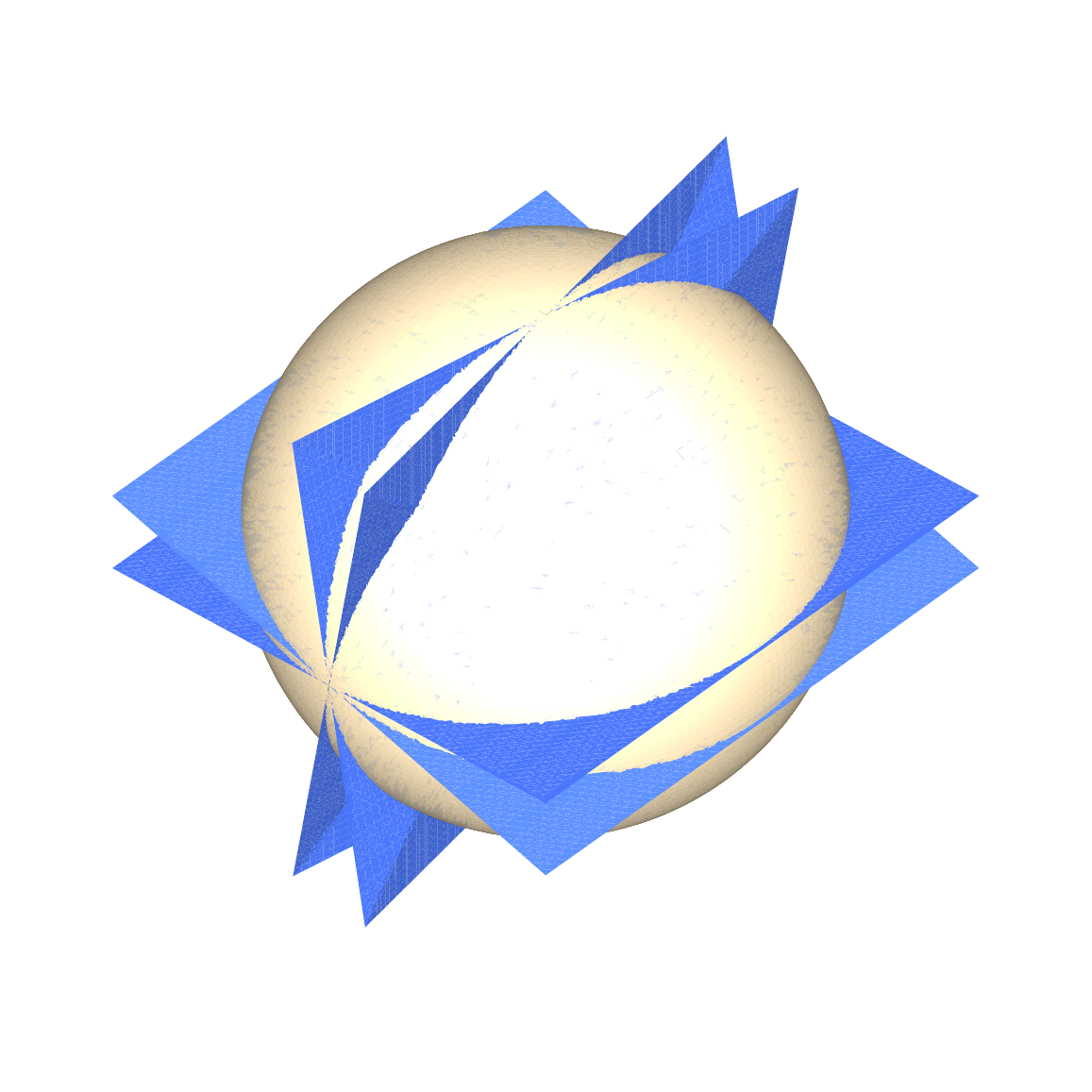}
\caption{\scriptsize Left: The cases (iii)\,a, (iii)\,b and (iii)\,c, in order from left to right. The illustrations show the qualitative intersection of the solution set \(\SS_{(iii)_{\alpha}}\), \(\alpha=a,b,c\), with the ball \(\BB\) of radius \(\varepsilon > 0\) in \((t_1,t_2,t_3)\)-space.} 
\label{fig:caseiii}
\end{figure}

\subsection{Case \eqref{case:iii}} The only possibilities for some $m_i$ to divide another $m_j$ are that $m_1=1$, in which case it divides both $m_2$ and 
$m_3$, or that $m_2$ divides $m_3$, or contrariwise. Without loss of generality, assume that 
$m_3\nmid m_2$. We then obtain the following alternatives:
\begin{enumerate}
\item[(a)] $m_1=1$, $m_2 \mid m_3$. Then $m_1 \mid m_2 \mid m_3$. An example is $(1, 2, 4)$.
\begin{align*}
&\Phi_1(0,t_2, t_3)=0,\\
 &\Phi_2(0,0,t_3)=0, \\ 
 &\Phi_3(0, 0,0)=0.
\end{align*}
In this case we combine the techniques from the cases $(i)$, $(ii)_a$ and $(ii)_b$, by expanding $\Phi_1$ in $t_1$, $\Phi_2$ using cylindral coordinates $(r_{1,2} \hat t_1, r_{1,2} \hat t_2, t_3)$, and $\Phi_3$ using spherical coordinates. The resulting determinant will be non-zero whenever $|t_2| \geq \delta |t_1|$ and $|t_3| \geq \delta |t|$, which may be reduced to
\[
\SS_{(iii)_a}: \quad  |t_3| \geq \delta |t_2| \geq \delta^2 |t_1|,
\]
for some $\delta \ll 1$. The system of equations
\begin{align*}
t_1 \Psi_1(t,\Lambda) &= 0,\\
r_{1,2} \Psi_2(t,\Lambda) &= 0,\\
r \Psi_3(t,\Lambda) &= 0,
\end{align*}
has no further non-trivial solutions.

\medskip

\item[(b)] $m_1=1$, $m_2 \nmid m_3$. An example is $(1, 4, 6)$. Using techniques as in the previous examples, one obtains
\begin{align*}
&\Phi_1(0,t_2, t_3)=0,\\
 &\Phi_2(0,0,t_3)=0, \\ 
 &\Phi_3(0, t_2, 0)=0,
\end{align*}
with
\[
\SS_{(iii)_b}: \quad  \min(|t_2|,|t_3|) \geq \delta |t_1|.
\]
\medskip

\item[(c)] $m_1 >1$, $m_2 \mid m_3$. An example is $(2, 3, 9)$. We have
 \begin{align*}
&\Phi_1(0,t_2, t_3)=0,\\
&\Phi_2(t_1,0,0)=\Phi_2(0,0,t_3)=0, \\ 
&\Phi_3(t_1, 0,0)=0.
\end{align*}
Here, the union of $\{|t_3| \geq \delta |t_2| \geq \delta^2 |t_1|\}$ and $\{|t_3| \geq \delta |t_2| \geq \delta^2 |t_3|\}$ gives 
\[
\SS_{(iii)_c}: \quad  |t_3| \geq \delta |t_2| \geq \delta^2 \min(|t_1|,|t_3|).
\]
\medskip

\item[(d)] $m_1>1$, $m_2 \nmid m_3$. An example is $(2, 15, 21)$. For
\begin{align*}
&\Phi_1(0,t_2, t_3)=0,\\
 &\Phi_2(t_1,0,0)=\Phi_2(0,0,t_3)=0, \\ 
 &\Phi_3(t_1, 0,0)=\Phi_3(0, t_2, 0)=0.
\end{align*}
the solution set is the union
\begin{align*}
&\{|t_3| \geq \delta |t_2| \geq \delta^2 |t_1|\}\\ 
&\cup \{|t_2| \geq \delta |t_3| \geq \delta^2 |t_1|\}\\ 
&\cup \{|t_3| \geq \delta |t_2| \geq \delta^2 |t_3|\}\\
&\cup \{\min(t_3,t_2) \geq \delta |t_1|\},
\end{align*}
and we obtain
\[
\SS_{(iii)_d}: \quad  \min(|t_2|,|t_3|) \geq \delta \min(|t_1|, \max(|t_2|,|t_3|)).
\]
\medskip
\end{enumerate}

\begin{figure}
\includegraphics[width=0.2\linewidth, height=0.20\linewidth]{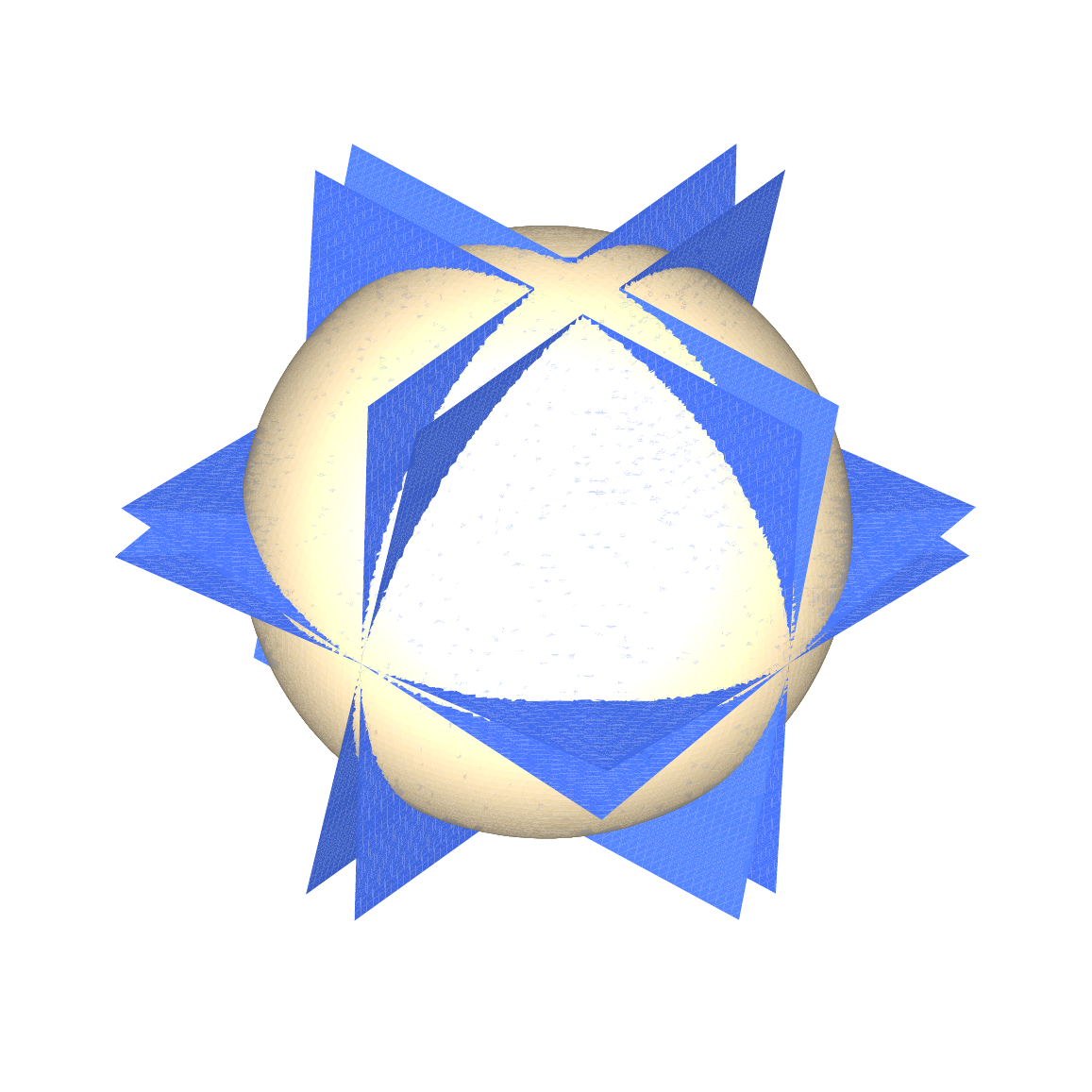}
\qquad
\includegraphics[width=0.2\linewidth, height=0.20\linewidth]{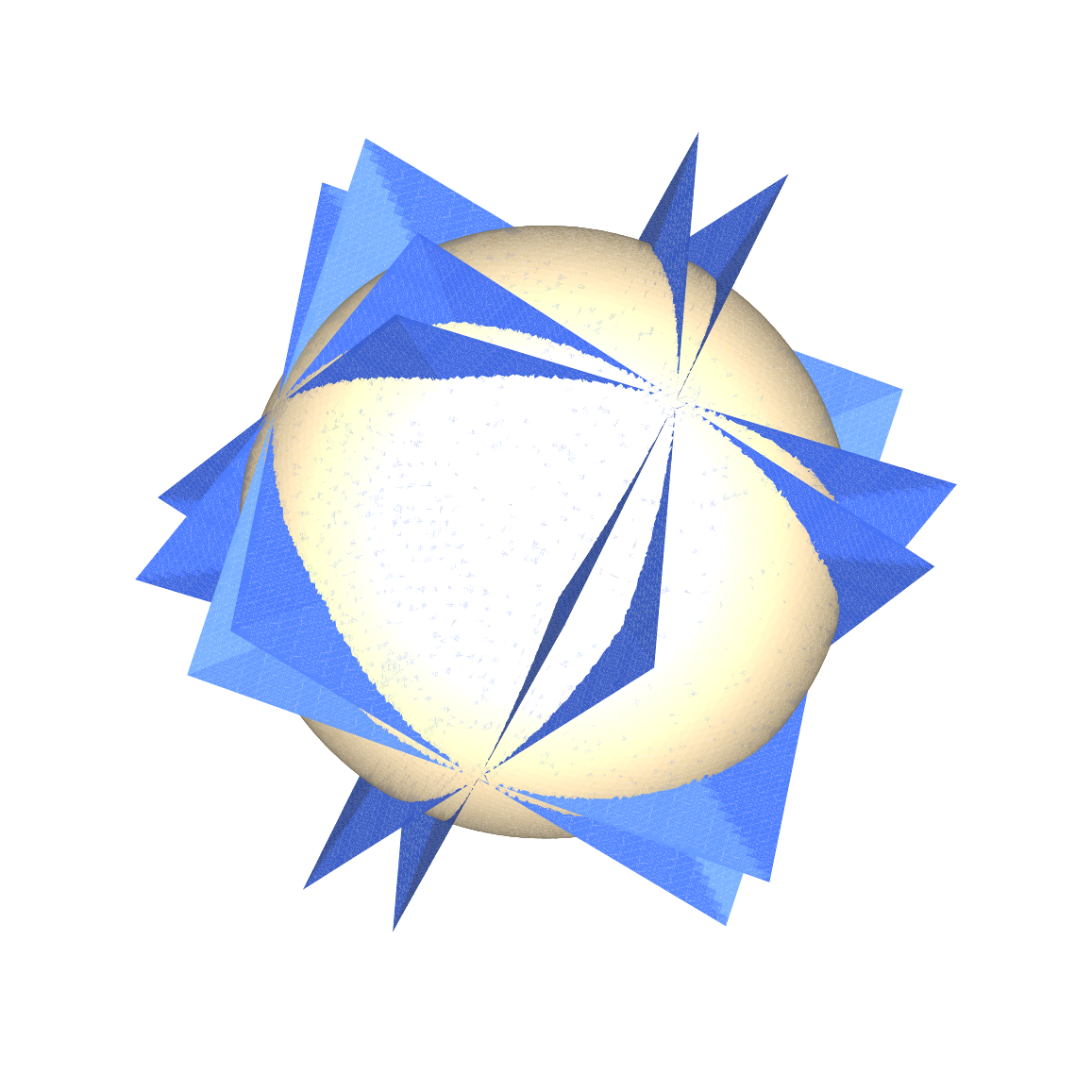}
\caption{Left: The cases (iv)\,a and (iv)\,b, in order from left to right. The illustrations show the qualitative intersection of the solution set \(\SS_{(ii)_{\alpha}}\), \(\alpha=a,b,c\), with the ball \(\BB\) of radius \(\varepsilon > 0\) in \((t_1,t_2,t_3)\)-space. Note the difference between to cases, which is visible only in one axial direction.} 
\label{fig:caseiv}
\end{figure}

\subsection{Case \eqref{case:iv}}
The only possiblities for some $m_i$ to divide another $m_j$ is that $m_i=1$. Without loss of generality, assume that $m_1=1$. 
There are two cases: 
\begin{enumerate}
\item[(a)] $m_1=1$. Then $m_1 \mid m_2, m_3$. An example is $(1, 2, 3)$. We have
\begin{align*}
&\Phi_1(0,t_2, 0)=\Phi_1(0,0,t_3)=0,\\
&\Phi_2(0,0,t_3)=0,\\
&\Phi_3(0, t_2, 0)=0,
\end{align*}
and
\[
\SS_{(iv)_a}: \quad  \min(|t_3|,|t_2|) \geq \delta |t_1| \geq \delta^2 \min(|t_3|,|t_2|).
\]

\medskip

\item[(b)] $m_1 \neq 1$. Then $m_1 \nmid m_2$, $m_1 \nmid m_3$, $m_2 \nmid m_3$. An example is $(2, 3, 5)$. The system
\begin{align*}
&\Phi_1(0,t_2, 0)=\Phi_1(0,0,t_3)=0,\\
 &\Phi_2(t_1,0,0)=\Phi_2(0,0,t_3)=0, \\ 
 &\Phi_3(t_1, 0,0)=\Phi_3(0, t_2, 0)=0.
\end{align*}
allows all eight  possible combinations of $|t_i| \geq \delta |t_j|$, with $i=1,2,3$ and $j \neq i$. Now, without loss of generality, say that $|t_1| \geq |t_2| \geq |t_3|$. Then also $|t_1| \geq \delta |t_2|$ and $|t_2| \geq \delta |t_3|$, so the only remaining condition is that
\[
|t_3| \geq \delta \min(|t_1|,|t_2|).
\]
Since this is to hold for arbitrary indices, one may let $|t_{j}| := \min_{i} |t_i|$  to obtain the uniform condition
\[
\SS_{(iv)_b}: \quad  |t_j|  \geq \delta \min_{i\neq j} |t_i|.
\]
Note that the set \(\SS_{(iv)_b}\) includes elements from the $t_j$-axes, for which $t_i = 0$ for $i \neq j$, although not the complete coordinate planes through the origin (for which only one $t_j = 0$).
\end{enumerate}
\bigskip

\begin{example}
For an illustration of our results, consider the wavenumbers \(k_1 = 6\), \(k_2 = 10\) and \(k_3 = 15\), as in case (i) in Section~\ref{subsec:anopenball}. According to Lemma~\ref{lemma:wavenumbers}, the bifurcation condition~\eqref{eq:bifurcation2} holds for these wavenumbers, and for no other wave numbers larger than \(6\). In fact, one can check that equality is obtained for \(\xi \approx 0.571\) and \(\alpha \approx -69.9\), with \(a \approx 7.65\), and that \(k=1,2,3,4,5\) do not satisfy the same bifurcation condition.  

The resulting kernel is  exactly three-dimensional and, via Theorem~\ref{thm:threedim} and the analysis pursued in Section~\ref{subsec:anopenball}, it gives rise to a full three-dimensional ball of solutions, attained in the horisontal direction as nonlinear perturbations of the linear hull of \(\cos(6q)\), \(\cos(10q)\) and \(\cos(15q)\). 

Figure~\ref{fig:threedimwaves} shows the envelope of three such waves; one where all coefficients are positive and equal, one where the coefficient in front of  the highest mode is zero (which gives a bimodal wave), and one where the coefficient in front of the highest mode is negative (but of the same size as those for the modes \(k_1=6\) and \(k_2=10\)). Note that although all of the waves constructed in this paper are symmetric around the axis \(q=0\), they are not necessarily so around their highest crest or lowest trough.
\end{example}

\begin{figure}
\includegraphics[width=0.8\linewidth]{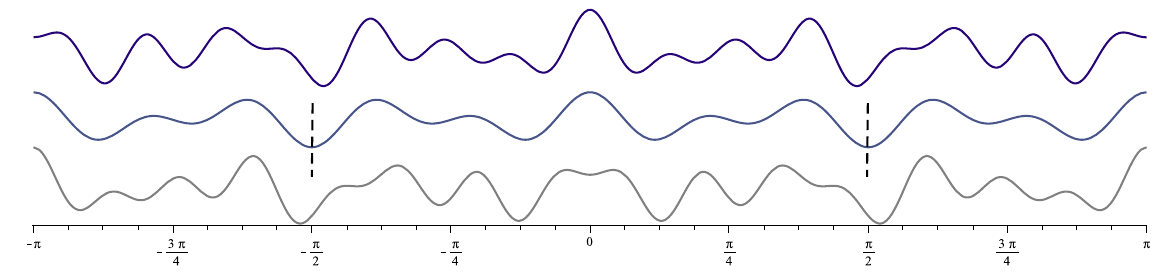}
\caption{\small Three examples of waves obtained as nonlinear perturbations of the linear hull of \(\cos(6q)\), \(\cos(10q)\) and \(\cos(15q)\). Note that although all of the waves constructed in this paper are symmetric around the axis \(q=0\), they are not necessarily so around their highest crest or lowest trough. Note also the wave in the middle is in fact a bimodal one---its minimal period is \(\pi\), not \(2\pi\), as for the two others.
}
\label{fig:threedimwaves}
\end{figure}

\noindent {\bf Acknowledgement.} The authors would like to thank the referee for several suggestions that helped improve the final form of the manuscript. ME furthermore acknowledges the support of the NRC grant \emph{Nonlinear Water Waves}, and EW acknowledges the support of the Swedish Research Council (grant no. 621-2012-3753).

\end{document}